\documentclass[a4paper,reqno,10pt,oneside]{amsart}
\usepackage{amsmath}
\usepackage[left=2.61cm,right=2.61cm,top=2.72cm,bottom=2.72cm]{geometry}
\usepackage[colorlinks=true,urlcolor=blue,
citecolor=red,linkcolor=blue,linktocpage,pdfpagelabels,
bookmarksnumbered,bookmarksopen]{hyperref}
\usepackage{amsmath, enumerate}
\usepackage{amsthm}
\usepackage{graphicx}
\usepackage{amssymb}
\usepackage{appendix}
\usepackage{mathtools}
\usepackage[utf8x]{inputenc}
\usepackage{fancyhdr}
\usepackage{calc}
\usepackage{url}
\usepackage[english]{babel}
\usepackage[T1]{fontenc}
\usepackage[usenames,dvipsnames]{color}
\usepackage{esint}

\allowdisplaybreaks[1]

\numberwithin{equation}{section}

\newcommand{\R}{\mathbb{R}}
\newcommand{\norm}[1]{\left\lVert#1\right\rVert}

\newcommand{\Acknowledgments}{\ \\\textbf{Acknowledgments:} }

\newtheorem{theorem}{Theorem}[section]
\newtheorem{proposition}[theorem]{Proposition}
\newtheorem{corollary}[theorem]{Corollary}
\newtheorem{lemma}[theorem]{Lemma}
\theoremstyle{definition}
\newtheorem{remark}[theorem]{Remark}
\theoremstyle{definition}

\theoremstyle{definition}

\begin{document}
	\author{Giulio Ciraolo}
	\address{G. Ciraolo. Dipartimento di Matematica "Federigo Enriques",
		Universit\`a degli Studi di Milano, Via Cesare Saldini 50, 20133 Milano, Italy}
	\email{giulio.ciraolo@unimi.it}
	
	\author{Alberto Farina}
	\address{A. Farina. LAMFA, UMR CNRS 7352, Universit\'e Picardie Jules Verne 33, rue St Leu, 
		80039 Amiens, France}
	\email{alberto.farina@u-picardie.fr}

	\author{Troy Petitt}
	\address{T. Petitt. Departamento de Matem\'{a}ticas,
		Universidad Carlos III de Madrid, Avda. de la Universidad, 30, 28911 Leganés, Spain}
	\email{tpetitt@math.uc3m.es}
	
	\begin{abstract}
		We study model semilinear equations on complete and non-compact weighted Riemannian manifolds with \emph{non-negative Bakry-\'{E}mery Ricci curvature}. Our main goal is to classify positive solutions of the equation at the Sobolev-critical exponent, and furthermore to prove that the existence of such solutions implies rigidity of the manifold and triviality of the weight. 
		This is possible when the weighted manifold has non-negative  finite dimensional Bakry-\'{E}mery Ricci curvature, and even under the weaker condition of non-negative infinite dimensional Bakry-\'{E}mery Ricci curvature, up to imposing some additional conditions in the latter case. To exhibit the sharpness of these additional conditions, we construct a non-trivial positive solution of the critical problem on a weighted manifold with positive infinite dimensional curvature.
		We also obtain a corresponding rigidity result for solutions of the Liouville equation on weighted Riemannian surfaces. Finally, we prove some non-existence theorems when the nonlinearity is sub-critical or simply under certain volume growth conditions. {In particular, the latter rules out all positive solutions on \emph{shrinking gradient Ricci solitons}.} 
		\end{abstract}
	\title{Rigidity of weighted manifolds via classification results for semilinear equations}
	\maketitle
	
	\section{Introduction}
	Let $(\mathcal{M},g)$ be a Riemannian manifold of dimension $d\in\mathbb{N}^*$ with $d\geq2$, let $p>1$, and let $u$ be a smooth, non-negative solution to the Lane-Emden equation 
	\begin{equation}\label{leq-unweighted}
		-\Delta u=u^p\quad\mathrm{in}\;\;\mathcal{M}\,.
	\end{equation}
	For $\mathcal{M}=\R^d$ with the Euclidean metric, a series of well-known papers classified non-negative solutions to this problem up to and including the \emph{Sobolev-critical} exponent 
	\begin{equation}\label{sobolev-critical}
		p_S(d):=
		\begin{cases}
			\tfrac{d+2}{d-2} \, \qquad  &\mathrm{for}\;\; d\geq3\,,\\
			\infty\,&\mathrm{for}\;\;d=2\,. 
		\end{cases}
	\end{equation}
	Indeed, for $p$ strictly less than $p_S(d)$ (or for all $p>1$ with $d=2$), it was proved in \cite{GS} that $u=0$ is the only such solution to \eqref{leq-unweighted}. For $d\geq3$, it was proved in \cite{CGS} that the set of positive solutions of \eqref{leq-unweighted} with $p=p_S(d)$ coincides with the two-parameter family defined by
	\begin{equation}\label{at-bubble}
		u(x)=\left(a+b\,|x-x_0|^2\right)^{-\frac{d-2}{2}}\quad\mathrm{for}\;\;a,b>0\;\;\mathrm{satisfying}\;\;1=d(d-2)ab\,,
	\end{equation}
	and any $x_0\in\R^d$. These special solutions are known as
	\emph{Aubin-Talenti bubbles}, and they were first identified in \cite{Aubin,Talenti} as the unique optimizers of the Sobolev inequality in $\R^d$ with sharp constant.  
	
	Already in \cite{GS}, the non-existence (or Liouville-type) result for sub-critical $p$ was extended to complete and non-compact Riemannian manifolds with non-negative Ricci curvature. However, positive solutions of \eqref{leq-unweighted} exist for all $p>1$ in the negatively curved hyperbolic space $\mathcal{M}=\mathbb{H}^d$ - see \cite{BGGV,MS}, so it is not in general possible to relax the non-negative curvature assumption. Returning to the case of non-negative curvature, there have been a few recent partial results concerning \eqref{leq-unweighted} with critical exponent $p=p_S(d)$ in \cite{CM,CFP,FMM}, which can be summarized by the following remarkable double rigidity statement: \emph{if there exists a solution $u>0$ to $-\Delta u=u^\frac{d+2}{d-2}$ in a complete, non-compact Riemannian manifold with non-negative Ricci curvature, then}\footnote{At the present time, this result has been proven only under some restrictions on the dimension $d$ or on the size of $u$.}
	\begin{equation*}
		(\mathcal{M},g) \;\;\mathrm{is}\;\;\mathrm{isometric}\;\;\mathrm{to}\;\;(\R^d,\mathrm{Eucl}_{\R^d})\quad\mathrm{and}\quad u\;\;\mathrm{is}\;\;\eqref{at-bubble}\,.
	\end{equation*}
	
	In dimension $d=2$, the corresponding rigidity result concerns solutions to the Liouville equation 
	\begin{equation}\label{lioueq-unweighted}
		-\Delta u=e^u\quad\mathrm{in}\;\;\mathcal{M}\,.
	\end{equation}
	In \cite{CFP}, the authors proved a sharp result for solutions of \eqref{lioueq-unweighted} under the assumption of non-negative {Gaussian} curvature. In this case, they find that if there exists a solution $u$ of \eqref{lioueq-unweighted} satisfying a certain optimal {asymptotic lower} bound, then $(\mathcal{M},g)$ is isometric to $(\R^2,\mathrm{Eucl}_{\R^2})$, and 
	\begin{equation}\label{log-bubble}
		u(x)=\log\frac{1}{\left(a+b\,|x-x_0|^2\right)^2} \quad\mathrm{for}\;\;a,b>0\;\;\mathrm{satisfying}\;\;1=8ab\,,
	\end{equation}
	for some $x_0\in\R^2$. {Similar results are contained in \cite{CL,CM} by also assuming the finite mass condition $\int e^u < + \infty$. We also mention the recent survey \cite{CL0} {for further references, historical comments, and connections to geometry}.}

	In this work, by suitably adapting the strategy of \cite{CFP}, we shall extend these results to the framework of weighted manifolds, which we introduce now. As before, let $(\mathcal{M},g)$ be a Riemannian manifold of dimension $d\geq2$. For some $f\in C^2(\mathcal{M})$, a weighted Riemannian manifold is the triple $(\mathcal{M},g,\mu)$, where $\mu$ is the measure with density
	\begin{equation*}
		d\mu=e^{-f}\,d\nu\,,
	\end{equation*}
	where $d\nu$ is the Riemannian volume element.
	The associated weighted Laplace operator is\footnote{In the below expression, and from now on, we take for granted that all metric and differential structures are defined with respect to the Riemannian metric $g$, so we choose to omit such references in the notations whenever possible.}
	\begin{equation}\label{def-L}
		\mathcal{L}w \, := \, \Delta w - \nabla f\cdot \nabla w \, = \, e^f\mathrm{div}\,(e^{-f}\,\nabla w) \, ,
	\end{equation}
	for any $w\in C^2(\mathcal{M})$. Clearly when $\nabla f\equiv0$, we reduce to the case of unweighted manifolds. 
	
	We are concerned with classifying non-negative solutions of the Lane-Emden equation
	\begin{equation}\label{leq}
		-\mathcal{L} u \, = \, u^p\, \quad\mathrm{in}\;\;\mathcal{M}\, \qquad\mathrm{for}\quad p>1\,,
	\end{equation}
	and solutions of the Liouville equation 
	\begin{equation}\label{liou-eq}
		-\mathcal{L}u=e^u \quad \mathrm{in}\;\; \mathcal{M}\, ,
	\end{equation}
	in dimension $d=2$.
	In view of the above discussion, such classification is made possible by imposing some Ricci curvature lower bounds on the weighted manifold. There are a few suitable definitions of weighted Ricci curvature in the literature, first appearing in the works \cite{Bakry1,Bakry2} of D. Bakry. Indeed, for any real number $n$ satisfying $n\geq d$, or $n=\infty$, we define the $n$-dimensional Bakry-\'{E}mery Ricci curvature tensor 
	\begin{equation}\label{def-BER}
		\mathrm{Ric}_{n,d}:=
		\begin{cases}
			\mathrm{Ric} + \nabla^2 f\quad &\mathrm{if} \;\; n = \infty\, , \\
			\mathrm{Ric} + \nabla^2 f - \frac{\nabla f \otimes \nabla f}{n-d} \quad &\mathrm{if} \;\; n > d \, , \\
			\mathrm{Ric} & \mathrm{if} \;\; n = d \,. \\
		\end{cases}
	\end{equation}
	We use the convention that $n=d$ if and only if $\nabla f\equiv 0$ , \emph{i.e.} the weight is trivial. 
	With this convention in mind, from \eqref{def-BER}, it follows that $\mathrm{Ric}_{n,d}\leq\mathrm{Ric}_{\infty,d}$ for all $n\geq d$, \emph{i.e.} $\mathrm{Ric}_{n,d}\geq0$ is stronger than $\mathrm{Ric}_{\infty,d}\geq0$. As we will see below, the {real} parameter $n\geq d$ related to Bakry-\'{E}mery Ricci curvature serves as a {virtual} dimension for the manifold $\mathcal{M}$, whose topological dimension is $d\in\mathbb{N}^*$.
	
	Since we will make extensive use of comparison methods, let us briefly summarize what is known when $\mathrm{Ric}$ is replaced by \eqref{def-BER} on weighted manifolds.
	When $(\mathcal{M},g,\mu)$ satisfies $\mathrm{Ric}_{n,d}\geq0$ for $n\geq d$, the standard comparison results (of Laplacian and Bishop-Gromov type) for $\mathrm{Ric}\geq0$ were successfully extended in \cite{BQ,Qian}. The corresponding theory for $\mathrm{Ric}_{\infty,d}\geq0$ is significantly more delicate and was established in \cite{WeiWylie}, up to imposing some further, unavoidable conditions. We summarize these results in Appendix \ref{appen:comp}. 
	Let us quickly note that Laplacian comparison always implies volume growth comparison with the same parameter, which for convenience we call $d$; see \cite{WeiWylie}. 
	So in particular, if we assume 
	\begin{equation}\label{lapl-comp-inf-eq}
		\mathcal{L}(\mathrm{dist}(x,o))\leq \frac{d-1}{\mathrm{dist}(x,o)}\quad\mathrm{weakly}\; \mathrm{in}\;\;\mathcal{M}\,,
	\end{equation}
	for some $o\in\mathcal{M}$,
	we implicitly assume in addition that volume comparison holds with respect to the same point $o$, and with the same $d$, that is, 
	\begin{equation}\label{bis-gro-inf}
		\int_{\mathcal{B}_R(o)}e^{-f}\,\leq C\,R^d\,,\quad\mathrm{for \,all}\;\;R\geq r\,,
	\end{equation}
	for all $r>0$, where $C>0$ depends only on $d$ and the geometry of $\mathcal{M}$ in $\mathcal{B}_r(o)$. {We also remark that condition \eqref{bis-gro-inf} is automatically satisfied when the potential function $f$ is bounded on $\mathcal M$ (see for instance \cite{MuWa} and \cite{Yang}).}
		
	\subsection{Notation} 
	As is quite standard, we may use the notation $\mathrm{dist}(x,o)$ or $|x-o|$ to denote the geodesic distance from a point $x\in\mathcal{M}$ to a fixed point $o\in\mathcal{M}$. Furthermore, we may also consider geodesic radii with symbols $R$, $r$, or $r(x,o)$ which again tacitly refer to the same geodesic distance when no confusion occurs. Finally, an expression like $u\in\mathcal{O}\!\left(|x|^{-\alpha}\right)$ means that $u\in\mathcal{O}\!\left(|x-o|^{-\alpha}\right)$ for some $o\in\mathcal{M}$.
	
	\section{Statement of the main results and paper organization}
	As discussed above, in the weighted manifold framework there are two valid types of weighted Ricci curvature that may replace the standard condition $\mathrm{Ric}\geq0$. First, we state our results involving the Lane-Emden equation \eqref{leq} and the infinite-dimensional Ricci curvature tensor defined in \eqref{def-BER}, and then we provide a discussion and some relevant examples. 
	Due to the notable lack of rigidity of weighted manifolds satisfying $\mathrm{Ric}_{\infty,d}\geq0$, in order to achieve the desired results, it turns out to be necessary (see Section \ref{sec:bubble}) to make some further assumptions on the weighted manifold and the solution: these are the comparison inequalities
	\eqref{lapl-comp-inf-eq} and/or \eqref{bis-gro-inf},
	and the pointwise condition
	\begin{equation}\label{dot-cond}
		\nabla f\cdot\nabla u\leq0\quad\mathrm{in}\;\;\mathcal{M}\,,
	\end{equation}
	where $u$ is the solution of \eqref{leq} or \eqref{liou-eq} under consideration.
	Subsequently, we state some corresponding theorems for the case of finite-dimensional Ricci curvature.
	Finally, we state our classification theorem for the Liouville equation \eqref{liou-eq}.
	
	\subsection{Results concerning the Lane-Emden equation under an infinity curvature condition}
	The following theorem should be interpreted as a rigidity and triviality result for weighted manifolds satisfying $\mathrm{Ric}_{\infty,d}\geq0$ that admit positive solutions of \eqref{leq-d-crit}. That is, if such a solution exists on a suitable weighted manifold, it follows that $u$ is a bubble \eqref{at-bubble}, the manifold is in fact the Euclidean space, and the weight is trivial (a constant).
	
	Concerning the conditions $i)$, $ii)$, and $iii)$ below, these are analogous to the ones used in \cite{CFP} in the unweighted case, which we emphasize are currently not known to be improvable.
	\begin{theorem}\label{thm-infty-crit}
		Let $(\mathcal{M},g,\mu)$ be a complete, connected, non-compact, boundaryless weighted Riemannian manifold of dimension $d \geq 3$. 
		Let $\mathrm{Ric}_{\infty,d}\geq0$ and let $u\in C^3(\mathcal{M})$ be a positive solution of 
		\begin{equation}\label{leq-d-crit}
			-\mathcal{L}u=u^{\frac{d+2}{d-2}} \quad \mathrm{in}\;\; \mathcal{M}\, ,
		\end{equation}
		where $f\in C^2(\mathcal{M})$ additionally satisfies \eqref{dot-cond},
		and also assume that \eqref{bis-gro-inf} holds with respect to some $o\in\mathcal{M}$.
		Assume that any of the three following conditions holds:
		\begin{enumerate}[i)]
			\item $d\in\left\{3,4,5\right\}$\,,
			\item $d\geq6$, \eqref{lapl-comp-inf-eq} holds, and
			\begin{equation}\label{pw-decay-cond-d}
				u\in\mathcal{O}\!\left(|x|^{-\alpha}\right)\;\;\mathrm{as}\;\;|x|\to+\infty\,,\qquad\mathrm{for}\;\;\alpha>\frac{(d-2)(d-6)}{4(d-4)}\,,
			\end{equation}
			\item $u$ has finite weighted energy, namely that
			$$\int_\mathcal{M}e^{-f}\,u^{\frac{2d}{d-2}}<+\infty\,,$$
			and \eqref{lapl-comp-inf-eq} holds.
		\end{enumerate}

		\noindent Then $(\mathcal{M},g)$ is isometric to $(\R^d,\mathrm{Eucl}_{\R^d})$ and $\nabla f\equiv0$, {i.e.} $\mathcal{L}=\Delta$. Furthermore, $u$ is
		precisely the Aubin-Talenti bubble \eqref{at-bubble}.
	\end{theorem}

\begin{remark}\label{rem:rough-lapl}
	Although for simplicity we include \eqref{lapl-comp-inf-eq} as a hypothesis in Theorem \ref{thm-infty-crit} \emph{ii)}, it is in fact only necessary to assume a significantly weaker pointwise control of the form
	\begin{equation}\label{rough-lapl-comp}
		\mathcal{L}(\mathrm{dist}(x,o))\leq \frac{\mathsf{C}}{\mathrm{dist}(x,o)}\quad\mathrm{weakly}\; \mathrm{in}\;\;\mathcal{M}\,,
	\end{equation}
	for some $o\in\mathcal{M}$ and \emph{any} $\mathsf{C}>0$, (the constant $\mathsf{C}>0$ in \eqref{rough-lapl-comp} can be much larger than the dimension $d$). Indeed, in the proof, assumption \eqref{lapl-comp-inf-eq} is only used via the Yau-type estimate found in Lemma \ref{yau-lem}, and the Laplacian comparison assumption in the statement of Lemma \ref{yau-lem} is precisely \eqref{rough-lapl-comp}. For Theorem \ref{thm-infty-crit} \emph{iii)} on the other hand, \eqref{lapl-comp-inf-eq} with the precise constant $d$ is needed.
\end{remark}
	
\begin{remark} 
	One may ask about the suitability of the additional conditions \eqref{lapl-comp-inf-eq} and/or \eqref{bis-gro-inf} and \eqref{dot-cond} used in the statement of Theorem \ref{thm-infty-crit}. As is well-known, the condition $\mathrm{Ric}_{\infty,d}\geq0$ is too weak by itself to obtain the expected Bishop-Gromov, Laplacian comparison, or Myers'-type results; see \cite[Section 2]{WeiWylie} for some counterexamples. Therefore it is quite reasonable to expect that in the current context, one must add some suitable conditions to reinforce the rigidity of the problem.
	
	From the technical point of view, \eqref{dot-cond} is the crucial condition that allows us to completely exploit the methods introduced in \cite{CFP}. However, in Section \ref{sec:bubble}, we prove that \eqref{dot-cond} is \emph{anything but a technical assumption}: in fact, we construct a family of weighted manifolds and solutions of \eqref{leq-d-crit} satisfying $\mathrm{Ric}_{\infty,d}\geq0$, \eqref{bis-gro-inf}, and \eqref{pw-decay-cond-d}--\eqref{rough-lapl-comp} \emph{but not \eqref{dot-cond}}, \emph{i.e.} valid counterexamples for Theorem \ref{thm-infty-crit} when \eqref{dot-cond} does not hold, for all $d\geq3$ (in view of Remark \ref{rem:rough-lapl} and the observation that \eqref{asy-upp-bd-u} implies \eqref{pw-decay-cond-d} for $p=p_S(d)$). See Theorem \ref{constr-sol-theo} below.
	
	As for the relation between \eqref{lapl-comp-inf-eq}--\eqref{bis-gro-inf} and \eqref{dot-cond}, in the radial case, \emph{i.e.} when $\mathcal{M}$, $f$, and $u$ are all rotationally symmetric about a pole $o\in\mathcal{M}$, \emph{\eqref{dot-cond} implies \eqref{lapl-comp-inf-eq}--\eqref{bis-gro-inf}}. This follows from \cite[Theorem 1.2]{WeiWylie} and the observation that such solutions are monotonically decreasing along geodesics starting from $o\in\mathcal{M}$. 
\end{remark}
			
We formalize the above remark in the following Theorem, whose proof is postponed until Section \ref{sec:bubble}.
\begin{theorem}\label{constr-sol-theo}
	For all $d\geq3$, $d\in\mathbb{N}$, there exists a smooth weighted Riemannian model manifold\footnote{Meaning that the metric and the weight are radially symmetric about some point.} $(\mathcal{M},g, \mu)$ of dimension $d$ with pole $o\in\mathcal{M}$ satisfying $\mathrm{Ric}_{\infty,d}\geq0$ in $\mathcal{M}$ and $\mathrm{Ric}_{\infty,d}>0$ in $\mathcal{M}\setminus \{o\}$ and admitting, for any $\ell>0$ and any $p\geq\tfrac{d+2}{d-2}$, a smooth, positive, radial solution $u$ of \eqref{leq} satisfying the following conditions:
	\begin{enumerate}[1)]
		\item $\mathcal{M}$ is complete, non-compact, boundaryless, and diffeomorphic to $\R^d$\,,
		\item Bishop-Gromov volume comparison \eqref{bis-gro-inf} holds\,,
		\item Laplacian comparison \eqref{lapl-comp-inf-eq} does not hold, but \eqref{rough-lapl-comp} does hold\,,
		\item $u(o)=\ell$\,,
		\item $u$ is decreasing along all geodesics starting from $o$,
		\item $\nabla f \cdot \nabla u> 0$ in $\mathcal{M}\setminus\{o\}$\,,
		\item $u$ satisfies
		\begin{equation}\label{asy-upp-bd-u}
			u(r)\leq \left(\frac{1}{C\,r^2+\ell^{1-p}}\right)^{\frac{1}{p-1}}\quad\mathrm{for}\;\mathrm{all}\;\;r>0\,,
		\end{equation} 
		for $C>0$.
	\end{enumerate}
\end{theorem}

We point out that the above Theorem furnishes a new example of a weighted manifold with positive curvature where Laplacian comparison \eqref{lapl-comp-inf-eq} does not hold. Whereas the example given in \cite[Example 2.3]{WeiWylie} is "pathological" in the sense of differing strongly from $\R^d$ (it has exponential weighted volume growth and lacks symmetry), our example exhibits the expected polynomial volume growth and full radial symmetry.
\normalcolor
	
	Next, let us discuss our non-existence results. The two following Liouville-type theorems are proven independently and have quite different hypotheses. First, Theorem \ref{thm-infty} requires a Ricci curvature condition, the condition \eqref{dot-cond}, \emph{and} a Bishop-Gromov volume growth condition, but is valid up to the classical Sobolev exponent \eqref{sobolev-critical}. On the other hand, Theorem \ref{thm:vol-growth-liou} holds under a \emph{sole} volume growth hypothesis, but is valid only up to a lower exponent.
	When $d=2$, the non-existence result holds under the sole assumption of \eqref{bis-gro-inf}, as follows from Theorem \ref{thm:vol-growth-liou}. Note that in this case, \eqref{bis-gro-inf} implies the main assumption of Theorem \ref{thm:vol-growth-liou}.
	\normalcolor  
	
	\begin{theorem}\label{thm-infty}
		Let $(\mathcal{M},g,\mu)$ be a complete, connected, non-compact, boundaryless weighted Riemannian manifold of dimension $d \geq 2$. 
		Let $\mathrm{Ric}_{\infty,d}\geq0$, and let $u\in C^3(\mathcal{M})$ be a non-negative solution of 
		\begin{equation*}
			-\mathcal{L}u=u^p\quad\mathrm{in}\;\;\mathcal{M}\,,\quad\mathrm{for}\;\mathrm{some}\;\;1<p<p_S(d)\,,
		\end{equation*}
		where $f\in C^2(\mathcal{M})$ additionally satisfies \eqref{dot-cond} with respect to $u$, and also assume that \eqref{bis-gro-inf} holds with respect to some $o\in\mathcal{M}$.
		Then $u\equiv0$.
	\end{theorem}
{In the following result, which is proven in Section \ref{sec:prelim:main}, we do not make any curvature assumption, we only assume a volume-growth condition.}
	
	\begin{theorem}\label{thm:vol-growth-liou}
		Let $(\mathcal{M},g,\mu)$ be a complete, connected, non-compact, boundaryless weighted Riemannian manifold of dimension $d\geq2$. Let $u\in C^3(\mathcal{M})$ be a non-negative solution of 
		\begin{equation*}
			-\mathcal{L}u=u^p\quad\mathrm{in}\;\;\mathcal{M}\,,
		\end{equation*}
		for some $p>1$. Then\footnote{On the right-hand side of the following equation, we are referring to \emph{little-o notation}.}, if $\mu(\mathcal{B}_{R}(o))=o\!\left(R^{\frac{2p}{p-1}}\right)$ for some $o\in\mathcal{M}$, it follows that $u\equiv0$.
	\end{theorem}

	In the special case that {$(\mathcal{M},g,\mu)$ satisfies {$\mu(\mathcal{B}_R(o))=\mathcal{O}(R^2 \log R)$}}, by the above theorem it admits no such positive solutions of \eqref{leq} for any $p>1$.
	{{This is consistent with \cite[Theorem 9.7]{Grigoryan2} (see also the works cited therein), which implies that $(\mathcal{M},g,\mu)$ is \emph{parabolic} if {$\mu(\mathcal{B}_R(o))=\mathcal{O}(R^2 \log R)$.}}}
	We recall that a weighted manifold is parabolic if and only if it does not support positive entire $\mathcal{L}$-superharmonic functions.

	{The preceding Theorem \ref{thm:vol-growth-liou}, a direct consequence of integral estimates proved in Section \ref{sec:main}, is in fact a generalization of \cite[Theorem 4.5]{CFP} to weighted manifolds (for the unweighted case we also refer to \cite{GrSu} where the authors consider non-negative supersolutions of $-\Delta u= u^p$). However, in the case of weighted manifolds, Theorem 4.5 attains a special significance thanks to its application to many relevant weighted manifolds. This is discussed in the next Example 1, while the subsequent Example 2 is devoted to the construction to a \emph{non-parabolic} weighted manifold for which Theorem \ref{thm:vol-growth-liou} applies.}
		
		\medskip
		
{\bf Example 1.} Any weighted manifold satisfying
		\begin{equation}\label{strict-ric}
			\mathrm{Ric}_{\infty,d}\geq\lambda\,,\quad\mathrm{for}\;\mathrm{some}\;\;\lambda>0\,,
		\end{equation}
		has finite weighted volume -- see \cite[Corollary 4]{Mor}. Therefore, Theorem \ref{thm:vol-growth-liou} tells us that there are no positive solutions to \eqref{leq} for any $p>1$ if \eqref{strict-ric} holds. 
		
		As a particular case, if $(\mathcal{M},g,\mu)$ satisfies
		\begin{equation*}
			\mathrm{Ric}_{\infty,d}=\lambda\,,\quad\mathrm{for}\;\mathrm{some}\;\;\lambda>0\,,
		\end{equation*}
		it is a \emph{shrinking gradient Ricci soliton}. 
		Such manifolds have been intensely studied, due in part to their relation to self-similar limits of Hamilton's Ricci flow. Some explicit examples are the Gaussian space\footnote{In the following two examples, we make a small abuse of notation by considering the \emph{density} rather than the measure in the third slot of the triple $(\mathcal{M},g,\mu)$.} $(\R^d,\mathrm{Eucl}_{\R^d},e^{-\alpha|x|^2}dx)$, the weighted hyperbolic space $(\mathbb{M}^d,g_{\mathbb{M}^d},e^{-\alpha|x|^2}dx)$ (for a suitable choice of $\alpha>0$), or the warped product given in \cite[p. 1]{PRS2}.  
		Let us also point out that, differently from the unweighted case, \eqref{strict-ric} \emph{does not} imply that $(\mathcal{M},g,\mu)$ is compact, as can be seen by the above examples. 
		If the weighted manifold is compact, a Liouville theorem for non-negative solutions follows for all $p$ simply by applying the strong maximum principle.

\medskip

{\bf Example 2.}  Here, we provide an example of a \emph{non-parabolic} weighted model manifold for which Theorem \ref{thm:vol-growth-liou} applies for all $p>1$, \emph{i.e.} there exists a positive $\mathcal{L}$-superharmonic function, but there \emph{do not} exist positive solutions to \eqref{leq}. 
		For any $d\geq2$, we consider $(\R^d,\mathrm{Eucl}_{\R^d},\mu)$ with $d\mu=e^{-f}\,dx$ for any smooth radial weight $f\equiv f(r)$ satisfying
		\begin{equation}\label{def-f}
			e^{-f(r)}= C\,r^{2-d}\,\log^\beta(r)\,,\quad\mathrm{for}\;\mathrm{all}\;\;r\gg 1\,,
		\end{equation}
		with $\beta>1$ and $C>0$. Calling $\mathcal{S}(r)=C_d\,r^{d-1}\,e^{-f(r)}$ the weighted surface area of $\partial \mathcal{B}_r$, we find from \eqref{def-f} that 
		\begin{equation*}
			\int^{+\infty}\frac{dr}{\mathcal{S}(r)}<+\infty\,,
		\end{equation*}
		which implies that the model manifold is not parabolic -- see \cite[Example 9.5]{Grigoryan2} and the sources cited therein. Furthermore 
		\begin{equation*}
			\mu(\mathcal{B}_R)=C\int_0^R\,r^{d-1}\,e^{-f(r)}\,dr\leq C\,R^2\,\log^\beta(R)\,,\quad\mathrm{for}\;\mathrm{all}\;\;R\gg1\,,
		\end{equation*}
		so it follows that $\mu(\mathcal{B}_R)=o\!\left(R^{\frac{2p}{p-1}}\right)$ and Theorem \ref{thm:vol-growth-liou} applies for all $p>1$.

			\subsection{Results concerning the Lane-Emden equation under the condition $\mathrm{Ric}_{n,d}\geq0$}
			{In the following two Theorems, we take the stronger assumption $\mathrm{Ric}_{n,d}\geq0$ and obtain weighted analogues of the main results of \cite{CFP} with a virtual dimension $n\geq d$.}
			\begin{theorem}\label{thm-leq-crit}
				Let $(\mathcal{M},g,\mu)$ be a complete, connected, non-compact, boundaryless weighted Riemannian manifold of dimension $d \geq 3$. Let $\mathrm{Ric}_{n,d}\geq0$ for some $n\geq d$, and let $u\in C^3(\mathcal{M})$ be a positive solution to 
				\begin{equation}\label{leq-crit}
					-\mathcal{L}u=u^{\frac{n+2}{n-2}} \quad \mathrm{in}\;\; \mathcal{M}\, .
				\end{equation}
				If any of the following four conditions holds:
				\begin{enumerate}[i)]
					\item $n\leq 5$\,,
					\item $5<n\leq6$ and 
					\begin{equation}\label{pw-decay-cond-1}
						u\in\mathcal{O}\!\left(|x|^{-\alpha}\right)\;\;\mathrm{as}\;\;|x|\to+\infty\,,\qquad\mathrm{for}\;\;\alpha>\frac{(n-2)(n-6)}{2(n-5)}\,,
					\end{equation}
					\item $n>6$ and
					\begin{equation}\label{pw-decay-cond}
						u\in\mathcal{O}\!\left(|x|^{-\alpha}\right)\;\;\mathrm{as}\;\;|x|\to+\infty\,,\qquad\mathrm{for}\;\;\alpha>\frac{(n-2)(n-6)}{4(n-4)}\,,
					\end{equation}
					\item $u$ has finite weighted energy, namely
					$$\int_\mathcal{M}e^{-f}\,u^{\frac{2n}{n-2}}<+\infty\,,$$
				\end{enumerate}
				then $(\mathcal{M},g)$ is isometric to $(\R^d,\mathrm{Eucl}_{\R^d})$ and $\nabla f\equiv0$, i.e. $\mathcal{L}=\Delta$, and $n=d$. Furthermore,
				$u$ is
				precisely the Aubin-Talenti bubble \eqref{at-bubble}.
			\end{theorem}
			Notice that from Theorem \ref{thm-leq-crit} $\emph{ii})$ we obtain a classification and rigidity result when $5<n<6$ and $u$ is \emph{bounded}. 
			
			Next, we state a Liouville-type theorem, which was proven independently in \cite[Corollary 1.10]{Lu} using the Bernstein-Yau method.
			\begin{theorem}\label{thm-subcrit}
				Let $(\mathcal{M},g,\mu)$ be a complete, connected, non-compact, boundaryless weighted Riemannian manifold of dimension $d \geq 2$. Let $\mathrm{Ric}_{n,d}\geq0$ for some $n\geq d$, and let $u\in C^3(\mathcal{M})$ be a non-negative solution to
				\begin{equation}\label{subcrit-p-cond}
					-\mathcal{L}u=u^p\quad\mathrm{in}\;\;\mathcal{M}\,,\quad\mathrm{for}\;\mathrm{some}\;\;1<p<p_S(n)\,.
				\end{equation}
				Then $u\equiv0$.
			\end{theorem}
			
			\subsection{On the Liouville equation}
			\begin{theorem}\label{thm:exp-class}
				Let $(\mathcal{M},g,\mu)$ be a complete, connected, non-compact, boundaryless weighted Riemannian surface\footnote{This is a weighted Riemannian manifold of dimension $d=2$.}. Let $\mathcal{G}(r)$ be any positive, non-decreasing function satisfying
				\begin{equation}\label{g-cond}
					\int_{R^*}^\infty\frac{dt}{t\,\mathcal{G}(t)}=+\infty\,,\quad\mathrm{for}\;\mathrm{some}\;\;R^*>0\,.
				\end{equation}
				Let $\mathrm{Ric}_{\infty,2}\geq0$ and let $u\in C^3(\mathcal{M})$ be a solution to 
				\begin{equation*}
					-\mathcal{L}u=e^u \quad \mathrm{in}\;\; \mathcal{M}\, ,
				\end{equation*}
				where $f\in C^2(\mathcal{M})$ additionally satisfies
				\eqref{dot-cond} with respect to $u$,
				and also assume that \eqref{bis-gro-inf} holds for some $o\in\mathcal{M}$. If $u$ satisfies
				\begin{equation}\label{pw-liou-cond}
					u(x)\geq-4\,\log\left[r(x,o)\,\mathcal{G}^{\frac{1}{2}}(r(x,o))\right]\,, \quad \mathrm{for}\;\mathrm{all}\;\;x\in \mathcal{M}\setminus \mathcal{B}_{R^*}(o)\,,
				\end{equation}
				then $(\mathcal{M},g)$ is isometric to $(\R^2,\mathrm{Eucl}_{\R^2})$ and $\nabla f\equiv0$, i.e. $\mathcal{L}=\Delta$. Furthermore, $u$ is
				precisely \eqref{log-bubble}.
			\end{theorem}
			The condition \eqref{g-cond} allows increasing functions behaving like $\mathcal{G}(r)\approx\log(r)$ or $\mathcal{G}(r)\approx\log(r)\log(\log r)$ at infinity.
			
			In accordance with the previous results, one may expect to be able to remove the assumption \eqref{dot-cond} by letting $\mathrm{Ric}_{n,2}\geq0$ for some $n>2$, but, as we explain in Section \ref{subsec:liou-proof} this is not possible using our method of proof.
			Also, it is explained in \cite[pp. 12--13]{CL} (see also \cite[Theorem 1.2]{CFP}) that the constant $4$ in \eqref{pw-liou-cond} is not improvable, even when $\nabla f\equiv0$.

			\subsection{Organization of the paper}
			The rest of the paper is organized as follows. In Section \ref{sec:bochner} we state some key pointwise equations and inequalities derived by applying Bochner's inequality to certain auxiliary functions that are defined in terms of the given solution $u$ of \eqref{leq} or \eqref{liou-eq}. In Section \ref{sec:sub-sub-proofs} we prove our main non-existence result for non-negative $\mathrm{Ric}_{n,d}$ curvature, Theorem \ref{thm-subcrit}. Section \ref{sec:main} is mostly dedicated to the proof of Theorem \ref{thm-leq-crit}. In order to prove this classification result, it is necessary to first derive some integral energy estimates -- Lemma \ref{int-est-lemma} -- from which Theorem \ref{thm:vol-growth-liou} trivially follows. 
			In Section \ref{sec:infty}, we consider $\mathrm{Ric}_{\infty,d}$ curvature, proving Theorem \ref{thm-infty-crit}, Theorem \ref{thm-infty}, and Theorem \ref{thm:exp-class}. Finally in Section \ref{sec:bubble}, we prove Theorem \ref{constr-sol-theo} by constructing some positive solutions of \eqref{leq} in weighted manifolds satisfying $\mathrm{Ric}_{\infty,d}\geq0$.
			
			\section{Some pointwise relations in third derivatives}\label{sec:bochner}
			Let $u$ be a positive solution of \eqref{leq}, and let
			$v=u^{-\frac{p-1}{2}}$. Notice that $v$ satisfies
			\begin{equation}\label{eq-for-v-0}
				\mathcal{L}v=\frac{1}{v}\left(\frac{m}{2}|\nabla v|^2+\frac{2}{m-2}\right)\,,
			\end{equation}
			where $m>2$ is defined by
			\begin{equation}\label{def-m}
				\frac{m}{2}=\frac{p+1}{p-1}
			\end{equation} 
			is a parameter that quantifies how much $p$ differs from the critical value $p_S(n)$ defined in \eqref{sobolev-critical}. Indeed, when $p\nearrow p_S(n)$, it follows $m\searrow n$. On the other hand, if $u$ is a solution of \eqref{liou-eq} in dimension $d=2$, and we set $v$ to satisfy 
			$u=-2\log v$, it holds
			\begin{equation}\label{eq-for-v-l}
				\mathcal{L}v=\frac{1}{v}\left(\left|\nabla v\right|^2+\frac{1}{2}\right)\,,
			\end{equation} 
			and we consider this as the case $d=m=2$ (formally).
			So, from now on, we combine \eqref{eq-for-v-0} and \eqref{eq-for-v-l} and simply consider positive solutions $v$ of 
			\begin{equation}\label{eq-for-v}
				\mathcal{L}v=\frac{1}{v}\left(\frac{m}{2}|\nabla v|^2+c_m\right)=:\mathtt{P}\,,
			\end{equation}
			for
			\begin{equation}\label{def-cm}
				c_m=\begin{cases}
					\frac{2}{m-2}\quad&\mathrm{if}\;\;m>2\,,\\
					\frac{1}{2}&\mathrm{if}\;\;m=2\,.
				\end{cases}
			\end{equation}
			We use the symbol $\mathtt{P}$ to refer to the quantity in \eqref{eq-for-v}.

			By the maximum principle for non-negative $\mathcal{L}$-superharmonic functions (which holds in our context of manifolds with non-negative Bakry-\'{E}mery Ricci curvature) either $u=0$ or $u>0$ everywhere for solutions of $-\mathcal{L}u=u^p$. Therefore, from now on, we assume that $u>0$ everywhere, so it follows by definition that $v>0$ and $\mathtt{P}>0$ everywhere (by \eqref{eq-for-v}). For the Liouville equation \eqref{liou-eq}, on the other hand, $u$ may change signs, but $v=(e^u)^{-\frac{1}{2}}$ is always strictly positive, and the same holds for $\mathtt{P}$. 
			
			If one suitably takes a further derivative of \eqref{eq-for-v} and applies Bochner's identity\footnote{Note that $\norm{\cdot}_{\mathrm{H.S.}}$ is the Hilbert-Schmidt norm of a matrix.}
			\begin{equation*}
				\tfrac{1}{2}\,\Delta\left|\nabla v\right|^2=\norm{\nabla^2 v}^2_{\mathrm{H.S.}}+\nabla v\cdot\nabla\Delta v +\mathrm{Ric}(\nabla v,\nabla v)\,,
			\end{equation*}
			or the $\mathcal{L}$-generalized form 
			\begin{equation}\label{BE-Bochner}
				\tfrac{1}{2}\,\mathcal{L}\left|\nabla v\right|^2=\norm{\nabla^2 v}^2_{\mathrm{H.S.}}+\nabla v\cdot\nabla\mathcal{L} v +\mathrm{Ric}_{\infty,d}(\nabla v,\nabla v)\,,
			\end{equation}
			the $\mathcal{L}$-weighted geometry of $\mathcal{M}$ appears in the form of the Bakry-\'{E}mery Ricci curvature tensor \eqref{def-BER}. Indeed, combining \cite[Lemma A.1]{CP}, where all the calculations can be found, with \eqref{eq-for-v} yields the fundamental pointwise equality
			\begin{equation}\label{ana-geo-eq}
				mv^{1-m}\,\mathsf{k}[v]=\mathrm{div}_{\!f}(v^{2-m}\,\nabla\mathtt{P})\,,
			\end{equation}
			where 
			\begin{equation}\label{def-k}
				\mathsf{k}[v]=\norm{\nabla^2v}_{\mathrm{H.S.}}^2-\frac{(\mathcal{L}v)^2}{m}+\mathrm{Ric}_{\infty, d}\,(\nabla v,\nabla v)\,,
			\end{equation}
			and 
			$$
			\mathrm{div}_{\!f}(\;\;)=e^{f}\mathrm{div}(e^{-f}\;\;) \,.
			$$ 
			In the case $\mathrm{Ric}_{n,d}\geq0$ for $n\geq d$, it can also be checked  that
			\begin{equation}\label{eq-k}
				\mathsf{k}[v]=\norm{\nabla^2v-\frac{\Delta v}{d}g\,}^2_{\mathrm{H.S.}}+\frac{m-n}{mn}\left(\mathcal{L}v\right)^2+\frac{n-d}{nd}\left(\mathcal{L}v+\frac{n}{n-d}\nabla f\cdot\nabla v\right)^{\!2}+\mathrm{Ric}_{n,d}\!\left(\nabla v,\nabla v\right),
			\end{equation}
			immediately implying that $k\geq0$ as long as $m\geq n$, that is, that $p\leq\frac{n+2}{n-2}$. For the case $n=\infty$, \eqref{eq-k} will be interpreted in Section \ref{sec:infty}.
			
			The overarching strategy to prove our classification results, \emph{i.e.} Theorem \ref{thm-infty-crit}, Theorem \ref{thm-leq-crit},  and Theorem \ref{thm:exp-class}, is to first prove that $\mathtt{P}\equiv c$ for some constant $c>0$, and then deduce some consequences from \eqref{ana-geo-eq} and \eqref{eq-k}. Indeed, if $\nabla \mathtt{P}\equiv 0$, then it follows from \eqref{ana-geo-eq} that $\mathsf{k}[v]=0$. Therefore, the four (non-negative) summands of \eqref{eq-k} must all be zero, which in turn implies rigidity of the manifold $\mathcal{M}$ and the solution $u$, and triviality of the weight $f$. This is an adaptation of the \emph{P-function method} employed in \cite{CFP}. However, in the case of Theorem \ref{thm-infty-crit} and Theorem \ref{thm:exp-class}, some modifications to the strategy must necessarily be made, due to the inherent lack of rigidity of weighted manifolds satisfying the weaker condition $\mathrm{Ric}_{\infty,d}\geq0$. 
			
			For the rest of the paper, equation \eqref{ana-geo-eq} is applied via the following pointwise inequality.
			\begin{lemma}
				Let $v\in C^3(\mathcal{M})$ be a positive solution to \eqref{eq-for-v}. Then for any $t\in\R$, it holds
				\begin{equation}\label{fund-pw-ineq-eq}
					\left(t-\tfrac{1}{2}\right)\mathtt{P}^{t-2}v^{2-m}|\nabla\mathtt{P}|^2+m\,\mathtt{P}^{t-1}v^{1-m}\,W_{\!f}[v]\leq \emph{div}_{\!f}\!\left(\mathtt{P}^{t-1}\,v^{2-m}\,\nabla\mathtt{P}\right),
				\end{equation}
				where
				\begin{equation}\label{def-w}
					W_{\!f}[v]=\begin{cases}
						\frac{m-d}{m^2}(\Delta v)^2+\mathrm{Ric}(\nabla v,\nabla v)&\mathrm{for}\;\;n=d\,,\\
						\tfrac{1}{m-d}\left(\tfrac{m-d}{m}\mathcal{L}v+\nabla f\cdot\nabla v\right)^2+\tfrac{m-n}{(n-d)(m-d)}\left|\nabla f\cdot\nabla v\right|^2+\mathrm{Ric}_{n,d}\,(\nabla v,\nabla v)&\mathrm{for}\;\;n>d\,,\\
						\tfrac{m-d}{m^2}(\mathcal{L}v)^2+\tfrac{2}{m}\mathcal{L}v\,\nabla f\cdot\nabla v+\mathrm{Ric}_{\infty,d}(\nabla v,\nabla v)&\mathrm{for}\;\;n=\infty\,.
					\end{cases}
				\end{equation}
			\end{lemma}
			
			\begin{proof}
				Let us begin by noticing that, by the definition of $\mathtt{P}$ in \eqref{eq-for-v}, we have
				\begin{equation*}
					\nabla\mathtt{P}=-\frac{1}{v^2}\,v\,\mathtt{P}\,\nabla v+\frac{m}{v}(\nabla^2v)\nabla v=\frac{m}{v}\left(\nabla^2v-\frac{\mathcal{L}v}{m}g\right)\nabla v\,.
				\end{equation*}
				By Cauchy-Schwartz inequality, it follows
				\begin{equation*}
					\left|\nabla\mathtt{P}\right|^2\leq\frac{m^2}{v^2}\norm{\nabla^2v-\frac{\mathcal{L}v}{m}g\,}_{\mathrm{H.S.}}^2\!\left|\nabla v\right|^2.
				\end{equation*}
				By \eqref{eq-for-v}, we see that
				\begin{equation*}
					|\nabla v|^2\leq\frac{2}{m}\mathtt{P}\,v\,,
				\end{equation*}
				so the above inequality becomes
				\begin{equation}\label{cs-grad-p}
					\left|\nabla\mathtt{P}\right|^2\leq\frac{2m}{v}\norm{\nabla^2v-\frac{\mathcal{L}v}{m}g\,}_{\mathrm{H.S.}}^2\!\mathtt{P}\,.
				\end{equation}
				In order to relate the squared norm appearing here to \eqref{ana-geo-eq} via \eqref{def-k}, we expand the norm and apply the definition of $\mathcal{L}$ in \eqref{def-L} to find that 
				\begin{equation*}
					\begin{aligned}
						\norm{\nabla^2v-\frac{\mathcal{L}v}{m}g\,}_{\mathrm{H.S.}}^2\!&=\norm{\nabla^2v}_{\mathrm{H.S.}}^2-\frac{2}{m}\mathcal{L}v\,\Delta v+\frac{d}{m^2}(\mathcal{L}v)^2\\
						&=\mathsf{k}[v]+\frac{(\mathcal{L}v)^2}{m}-\mathrm{Ric}_{\infty,d}(\nabla v,\nabla v)-\frac{2}{m}\mathcal{L}v\,\Delta v+\frac{d}{m^2}(\mathcal{L}v)^2\\
						&=\mathsf{k}[v]-\frac{m-d}{m^2}(\mathcal{L}v)^2-\frac{2}{m}\mathcal{L}v\,(\nabla f\cdot\nabla v)-\frac{1}{n-d}\left|\nabla f\cdot\nabla v\right|^2-\mathrm{Ric}_{n,d}(\nabla v,\nabla v)\\
						&=\mathsf{k}[v]-W_{\!f}[v]\,.
					\end{aligned}  
				\end{equation*}
				The last equality follows by completing the square and recalling the definition of $W_{\!f}$ in \eqref{def-w}.
				Substituting this back into \eqref{cs-grad-p} yields
				\begin{equation*}
					\left|\nabla\mathtt{P}\right|^2\leq\frac{2m}{v}\left(\mathsf{k}[v]-W_{\!f}[v]\right)\,\mathtt{P}.
				\end{equation*}
				Now we multiply the above inequality by $\mathtt{P}^{t-2}\,v^{2-m}$ and apply \eqref{ana-geo-eq} to the right-hand side to obtain
				\begin{equation*}
					\mathtt{P}^{t-2}\,v^{2-m}\left|\nabla\mathtt{P}\right|^2+2m\,\mathtt{P}^{t-1}\,v^{1-m}\,W_{\!f}[v]\leq 2\,\mathtt{P}^{t-1}\,\mathrm{div}_{\!f}(v^{2-m}\,\nabla\mathtt{P})\,.
				\end{equation*}
				The thesis follows by noticing that
				\begin{equation*}
					\mathrm{div}_{\!f}(\mathtt{P}^{t-1}\,v^{2-m}\,\nabla\mathtt{P})=\mathtt{P}^{t-1}\,\mathrm{div}_{\!f}(v^{2-m}\,\nabla\mathtt{P})+(t-1)\,\mathtt{P}^{t-2}\,v^{2-m}|\nabla \mathtt{P}|^2\,,
				\end{equation*}
				and rearranging terms.
			\end{proof}
			
			\section{A Liouville-type theorem for the subcritical Lane-Emden equation}\label{sec:sub-sub-proofs}
			From now on, we frequently use the notation $\mathcal{B}_R(o)\subset\mathcal{M}$ to denote the geodesic ball of radius $R>0$ centered at $o\in\mathcal{M}$ and 
			\begin{equation*}
				\mathcal{A}_R(o)=\mathcal{B}_{2R}(o)\setminus\overline{\mathcal{B}_{R}(o)}\,,
			\end{equation*}
			to denote the annulus.
			In Appendix \ref{app:yau}, we state and prove some standard results about cutoff functions, which are used in the sequel to localize computations inside $\mathcal{B}_R$ and $\mathcal{A}_R$.
			
			\begin{proof}[Proof of Theorem \ref{thm-subcrit}]
				Let us begin by fixing some notation. By the hypotheses $p<p_S(n)$ in \eqref{subcrit-p-cond}, we have $m>n$. Also, let us assume from now on that $n>d$, which is equivalent to $\nabla f\neq0$. Indeed, if $n=d$, we reduce to the case of \cite[Theorem 1.4]{CFP}.
				
				We aim to show that $\mathtt{P}\equiv0$, which is a contradiction with $v>0$ by \eqref{eq-for-v}. Finally, let us recall that, by definition, $\mathcal{L}v=\mathtt{P}$.
				
				Now we begin our calculations by studying \eqref{fund-pw-ineq-eq}--\eqref{def-w}. Since by assumption, $\mathrm{Ric}_{n,d}\geq0$, it follows that both sides of \eqref{fund-pw-ineq-eq} are non-negative, provided that
				\begin{equation}\label{t-cond}
					t > \frac{1}{2}\,.
				\end{equation} 
				Let $\varphi_R$ be a cutoff function as in Lemma \ref{cutoff-lem} (centered at any point $o\in\mathcal{M}$ that will not play a role in the proof). Now, we multiply all terms of \eqref{fund-pw-ineq-eq} by $e^{-f}\,\varphi_R^\theta$ for some $\theta>1$ to be chosen later and integrate by parts on the right-hand side. For the sake of simplicity, we denote $\varphi=\varphi_R$ until the end of the proof. We obtain
				\begin{equation*}
					\begin{aligned}
						\left(t-\tfrac{1}{2}\right)\!\int_{\mathcal{M}} e^{-f}\,\mathtt{P}^{t-2}v^{2-m}\,|\nabla\mathtt{P}|^2\,\varphi^\theta&+m\!\int_{\mathcal{M}} e^{-f}\,\mathtt{P}^{t-1}v^{1-m}\,W_{\!f}[v]\,\varphi^\theta\\
						&\leq -\theta\int_{\mathcal{M}} e^{-f}\,\mathtt{P}^{t-1}\,v^{2-m}\,(\nabla\mathtt{P}\cdot\nabla\varphi)\,\varphi^{\theta-1}\,.    
					\end{aligned}
				\end{equation*}
				Applying Young's inequality to the right-hand side in the form
				\begin{equation*}
					\left(\mathtt{P}^{\frac{t-2}{2}}|\nabla\mathtt{P}|\,\varphi^{\frac{\theta}{2}}\right)\left(\theta\,\mathtt{P}^{\frac{t}{2}}\,|\nabla \varphi|\,\varphi^{\frac{\theta-2}{2}}\right)\leq \varepsilon\,\mathtt{P}^{t-2}|\nabla\mathtt{P}|^2\,\varphi^{\theta}+\tfrac{\theta^2}{\varepsilon}\,\mathtt{P}^{t}\,|\nabla \varphi|^2\,\varphi^{\theta-2}\,,
				\end{equation*}
				gives
				\begin{equation*}
					\begin{aligned}
						\left(t-\tfrac{1}{2}-\varepsilon\right)\!\int_{\mathcal{M}} e^{-f}\,\mathtt{P}^{t-2}\,v^{2-m}\,|\nabla\mathtt{P}|^2\,\varphi^\theta&+m\!\int_{\mathcal{M}} e^{-f}\,\mathtt{P}^{t-1}\,v^{1-m}\,W_{\!f}[v]\,\varphi^\theta\\
						&\leq \frac{\theta^2}{\varepsilon}\!\int_{\mathcal{M}} e^{-f}\,\mathtt{P}^{t}\,v^{2-m}\,|\nabla \varphi|^2\,\varphi^{\theta-2}\,.    
					\end{aligned}
				\end{equation*}
				By \eqref{def-w} and the non-negativity of $\mathrm{Ric}_{n,d}$, we may rewrite the above inequality in the following way: 
				\begin{equation}\label{young-1}
					\begin{aligned}
						&\left(t-\tfrac{1}{2}-\varepsilon\right)\!\int_{\mathcal{M}} e^{-f}\,\mathtt{P}^{t-2}\,v^{2-m}\,|\nabla\mathtt{P}|^2\,\varphi^\theta\\&+\frac{m}{m-d}\int_{\mathcal{M}} e^{-f}\,\mathtt{P}^{t-1}\,v^{1-m}\left(\left(\tfrac{m-d}{m}\mathcal{L}v+\nabla f\cdot\nabla v\right)^2+\tfrac{m-n}{n-d}\left|\nabla f\cdot\nabla v\right|^2\right)\varphi^\theta\\
						&\leq \frac{\theta^2}{\varepsilon}\!\int_{\mathcal{M}} e^{-f}\,\mathtt{P}^{t}\,v^{2-m}\,|\nabla \varphi|^2\,\varphi^{\theta-2}\,.    
					\end{aligned}
				\end{equation}
				In preparation for another application of Young's inequality, this time to the term $\mathtt{P}^t$ on the right-hand side, we notice that 
				\begin{equation*}
					\begin{aligned}
						\mathtt{P}^2=\left[\tfrac{m}{m-d}\left(\tfrac{m-d}{m}\mathtt{P}+\nabla f\cdot\nabla v-\nabla f\cdot\nabla v\right)\right]^2
						\leq\tfrac{2m^2}{(m-d)^2}\left(\left(\tfrac{m-d}{m}\mathtt{P}+\nabla f\cdot\nabla v\right)^2+\left|\nabla f\cdot\nabla v\right|^2\right)\,,
					\end{aligned}
				\end{equation*}
				so, indeed, using the above inequality and Young's inequality with exponents $\frac{t+1}{t}$ and $t+1$, we find
				\begin{equation*}
					\begin{aligned}
						\mathtt{P}^t\,&v^{2-m}\,|\nabla \varphi|^2\,\varphi^{\theta-2}= \left(\mathtt{P}^{(t-1)\frac{t}{t+1}}\,\mathtt{P}^{\frac{2t}{t+1}}\,v^{(1-m)\frac{t}{t+1}}\,\varphi^{\frac{\theta t}{t+1}}\right)\!\left(v^{2-m-(1-m)\frac{t}{t+1}}\left|\nabla\varphi\right|^2\,\varphi^{\theta-2-\frac{\theta t}{t+1}}\right)\\
						&\leq \left(\left(\tfrac{2m^2}{(m-d)^2}\right)^{\frac{t}{t+1}}\mathtt{P}^{(t-1)\frac{t}{t+1}}\left(\left(\tfrac{m-d}{m}\,\mathtt{P}+\nabla f\cdot\nabla v\right)^2+\left|\nabla f\cdot\nabla v\right|^2\right)^{\frac{t}{t+1}}\,v^{(1-m)\frac{t}{t+1}}\,\varphi^{\frac{\theta t}{t+1}}\right)\\
						&\quad\times\!\left(v^{2-m-(1-m)\frac{t}{t+1}}\left|\nabla\varphi\right|^2\,\varphi^{\theta-2-\frac{\theta t}{t+1}}\right)\\
						&\leq \frac{\varepsilon^2}{\theta^2}\left(\mathtt{P}^{t-1}\left(\left(\tfrac{m-d}{m}\,\mathtt{P}+\nabla f\cdot\nabla v\right)^2+\left|\nabla f\cdot\nabla v\right|^2\right)\,v^{1-m}\,\varphi^{\theta}\right)\\
						&\quad + \mathtt{K}\,v^{2-m+t}\,\left|\nabla\varphi\right|^{2(t+1)}\,\varphi^{\theta-2(t+1)}\,,
					\end{aligned}
				\end{equation*}
				where $\varepsilon>0$ is arbitrary and $\mathtt{K}>0$ is a constant depending only on $d$, $m$, $t$, $\theta$, and $\varepsilon$. If we take 
				\begin{equation*}
					\varepsilon<t-\frac{1}{2}\,\land\,\frac{m}{m-d}\,\land\,\frac{m}{m-d}\frac{m-n}{n-d}\,,
				\end{equation*}
				and plug the above inequality into \eqref{young-1}, we obtain
				\begin{equation}\label{young-2}
					\begin{aligned}
						\left(t-\tfrac{1}{2}-\varepsilon\right)\!\int_{\mathcal{M}} e^{-f}\,\mathtt{P}^{t-2}\,v^{2-m}\,|\nabla\mathtt{P}|^2\,\varphi^\theta&+\left(\tfrac{m}{m-d}-\varepsilon\right)\int_{\mathcal{M}} e^{-f}\,\mathtt{P}^{t-1}\,v^{1-m}\left(\tfrac{m-d}{m}\mathcal{L}v+\nabla f\cdot\nabla v\right)^2\,\varphi^\theta\\
						&+\left(\tfrac{m}{m-d}\tfrac{m-n}{n-d}-\varepsilon\right)\int_{\mathcal{M}} e^{-f}\,\mathtt{P}^{t-1}\left|\nabla f\cdot\nabla v\right|^2\,v^{1-m}\,\varphi^\theta\\
						&\leq \mathtt{K}\int_{\mathcal{M}} e^{-f}\,v^{2-m+t}\,\left|\nabla\varphi\right|^{2(t+1)}\,\varphi^{\theta-2(t+1)}\,.    
					\end{aligned}
				\end{equation}
				Let us choose $t=m-\frac{3}{2}$, (which in view of the hypothesis that $m>n>2$ is a valid choice in the sense of \eqref{t-cond}), and $\theta$ sufficiently large, to obtain from the right-hand side of \eqref{young-2}
				\begin{equation}\label{young-3}
					\begin{aligned}
						\int_{\mathcal{M}} e^{-f}\,v^{\frac{1}{2}}\,\left|\nabla\varphi\right|^{2m-1}&\leq \norm{v}_{L^\infty(\mathcal{A}_R(o))}^{\frac{1}{2}}    \int_{\mathcal{M}} e^{-f}\,\left|\nabla\varphi\right|^{2m-1}\\
						&\leq C\,R\,R^{1-2m}\,R^n\,,
					\end{aligned}
				\end{equation}
				where $C>0$ depends only on $d$, $n$, the behavior of $v$ in $\mathcal{B}_1(o)$, and the geometry of $\mathcal{M}$ in $\mathcal{B}_1(o)$. In the previous estimate, we have used \eqref{def-cutoff} and \eqref{bis-gro}, as well as the estimate
				\begin{equation}\label{v-harm-appl-0}
					v(x)\leq C\,\mathrm{dist}(x,o)^{2\frac{n-2}{m-2}}\quad\mathrm{in}\;\;\mathcal{M}\setminus \mathcal{B}_1(o)\,,
				\end{equation}
				for a constant $C>0$ depending only on $n$ and the behavior of $v$ in $\mathcal{B}_1(o)$, which follows from Lemma \ref{super-harm-lem}. In fact, in the application of \eqref{v-harm-appl-0}, we have taken in particular the estimate with power $2$ (\emph{i.e.} $v^{\frac{1}{2}}\leq CR$), which is justified by $m>n$, and as we will see now is sufficient to conclude.
				
				Combining \eqref{young-2} and \eqref{young-3} yields
				\begin{equation}\label{young-4}
					\begin{aligned}
						\left(t-\tfrac{1}{2}-\varepsilon\right)\!\int_{\mathcal{M}} e^{-f}\,\mathtt{P}^{t-2}\,v^{2-m}\,|\nabla\mathtt{P}|^2\,\varphi^\theta&+\left(\tfrac{m}{m-d}-\varepsilon\right)\int_{\mathcal{M}} e^{-f}\,\mathtt{P}^{t-1}\,v^{1-m}\left(\tfrac{m-d}{m}\mathcal{L}v+\nabla f\cdot\nabla v\right)^2\,\varphi^\theta\\
						&+\left(\tfrac{m}{m-d}\tfrac{m-n}{n-d}-\varepsilon\right)\int_{\mathcal{M}} e^{-f}\,\mathtt{P}^{t-1}\left|\nabla f\cdot\nabla v\right|^2\,v^{1-m}\,\varphi^\theta\\
						&\leq C\,R^{-2m+n+2}\,,    
					\end{aligned}
				\end{equation}
				for all $R\geq1$, where $C>0$ depends only on $d$, $n$, $m$, the behavior of $v$ in $\mathcal{B}_1(o)$, and the geometry of $\mathcal{M}$ in $\mathcal{B}_1(o)$. We may take $R\to+\infty$ in \eqref{young-4} to conclude that the left-hand side is equal to $0$ if 
				\begin{equation*}
					m>\frac{n+2}{2}\,,
				\end{equation*}
				which clearly follows from $m>n$ and $n>2$.
				
				In this way, we deduce that \emph{both} of the following statements hold:
				\begin{gather}
					\mathrm{either}\;\;\mathtt{P}\equiv0\quad\mathrm{or}\quad \tfrac{m-d}{m}\,\mathtt{P}+\nabla f\cdot\nabla v=0\,\label{g1}\\
					\mathrm{either}\;\;\mathtt{P}\equiv0\quad\mathrm{or}\quad \nabla f\cdot\nabla v=0\,\label{g2}.
				\end{gather}
				It is easy to see (in particular since $m>d$) that the combination of \eqref{g1}--\eqref{g2} implies that $\mathtt{P}\equiv 0$, which, as we mentioned before, is a contradiction with the assumption that $v>0$, implying $v=0$. This in turn is a contradiction with the assumption that $u>0$ everywhere, which by the maximum principle implies $u=0$.
			\end{proof}
			
			\section{Classification results for the critical Lane-Emden equation}\label{sec:main}
			\subsection{Some preliminary results}\label{sec:prelim:main}
			\begin{lemma}[An integration-by-parts formula]\label{ibp-lem}
				Let ($\mathcal{M},g,\mu$) be a complete
				weighted Riemannian manifold of dimension $d\geq2$, and let $v\in C^3(\mathcal{M})$ be a positive solution to \eqref{eq-for-v}. Then, for any $q\in \R$ and $\psi\in C^{0,1}_c(\mathcal{M})$ it holds 
				\begin{equation}\label{ibp-formula-1}
					\left(\tfrac{m}{2}+1-q\right)\!\int_{\mathcal{M}}e^{-f}\,v^{-q}\,\left|\nabla v\right|^2\,\psi+c_m\int_{\mathcal{M}} e^{-f}\,v^{-q}\,\psi
					=-\int_{\mathcal{M}} e^{-f}\,v^{1-q}\,(\nabla v\cdot\nabla \psi)\,,
				\end{equation}
				where $c_m$ is defined in \eqref{def-cm}.
			\end{lemma}
			\begin{proof}
				It suffices to multiply \eqref{eq-for-v} by $e^{-f}\,v^{1-q}\,\psi$ and integrate one time by parts.
			\end{proof}
			
			\begin{corollary}[Corollary to Lemma \ref{ibp-lem}]\label{ibp-cor}
				Let ($\mathcal{M},g,\mu$) be a complete weighted
				Riemannian manifold of dimension $d\geq2$, and let $v\in C^3(\mathcal{M})$ be a positive solution to \eqref{eq-for-v}. Then for any $q,\ell\in\R$ and $\psi\in C^{0,1}_c(\mathcal{M})$ it holds
				\begin{equation}\label{ibp-formula-2}
					\begin{aligned}
						\left(\tfrac{m}{2}+1-q\right)\!\int_{\mathcal{M}}e^{-f}\,\mathtt{P}^\ell\,v^{-q}\,\left|\nabla v\right|^2\,\psi&+c_m\int_{\mathcal{M}} e^{-f}\,\mathtt{P}^\ell\,v^{-q}\,\psi\\
						&=-\int_{\mathcal{M}} e^{-f}\,\mathtt{P}^\ell\,v^{1-q}\,(\nabla v\cdot\nabla \psi)-\ell\!\int_{\mathcal{M}} e^{-f}\,\mathtt{P}^{\ell-1}\,v^{1-q}\,(\nabla v\cdot \nabla \mathtt{P})\,\psi\,,
					\end{aligned}
				\end{equation}
				where $c_m$ is defined in \eqref{def-cm}.
			\end{corollary}
			\begin{proof}
				Take $\psi\mapsto\mathtt{P}^\ell\,\psi$ in \eqref{ibp-formula-1}
			\end{proof}
		
			\begin{lemma}[Integral estimates]\label{int-est-lemma}
				Let $(\mathcal{M},g,\mu)$ be a complete weighted
				Riemannian manifold of dimension $d\geq2$. 
				Let $v\in C^3(\mathcal{M})$ be a positive solution to \eqref{eq-for-v}. Then the following estimates hold for all $o\in\mathcal{M}$ and all $R> 0$:
				\begin{enumerate}[i)]
					\item If $2\leq q< \frac{m}{2}+1$, then
					\begin{equation}\label{int-est-1-zero}
						\int_{\mathcal{B}_R(o)}e^{-f}\,v^{-q}\,\left|\nabla v\right|^2+\int_{\mathcal{B}_R(o)}e^{-f}\,v^{-q}\leq C\,\mu(\mathcal{B}_{2R}(o))\,R^{-q}\,.
					\end{equation}
					\item If $0\leq q\leq \frac{m}{2}+1$\, then 
					\begin{equation}\label{int-est-2-zero}
						\int_{\mathcal{B}_R(o)}e^{-f}\,v^{-q}\,\leq C\,\mu(\mathcal{B}_{2R}(o))\,R^{-q}\,,
					\end{equation}
				\end{enumerate}
				where $\mu(\Omega)=\int_\Omega d\mu=\int_{\Omega}e^{-f}\,dx$.
				The constant $C>0$ depends only on $d$, $m$, and $q$.
			\end{lemma}
			
			\begin{proof}
				\medskip
				\noindent{\textbf{Part i):}} Let us fix $o\in\mathcal{M}$ and $R>0$ as in the statement of the lemma, and consider \eqref{ibp-formula-1} with $\psi\mapsto\varphi_R^\theta$ for $\varphi_R$ as in Lemma \ref{cutoff-lem} and $\theta>1$ to be chosen later. 
				As before, for simplicity we continue to call $\varphi_R=\varphi$. 
				We have
				\begin{equation}\label{ibp-theta-ineq}
					\left(\tfrac{m}{2}+1-q\right)\!\int_{\mathcal{M}}e^{-f}\,v^{-q}\,\left|\nabla v\right|^2\,\varphi^\theta+c_m\int_{\mathcal{M}} e^{-f}\,v^{-q}\,\varphi^\theta
					\leq\theta\int_{\mathcal{M}} e^{-f}\,v^{1-q}\,\,\varphi^{\theta-1}\,|\nabla v|\,|\nabla \varphi|\,.
				\end{equation}
				Now, we use Young's inequality (twice) in the form
				\begin{equation*}
					\begin{aligned}
						v^{1-q}\,\,\varphi^{\theta-1}\,|\nabla v|\,|\nabla \varphi|&\leq\varepsilon\,v^{-q}\,\left|\nabla v\right|^2\,\varphi^\theta+\tfrac{1}{\varepsilon}\,v^{2-q}\,|\nabla\varphi|^2\,\varphi^{\theta-2}\\
						&\leq \varepsilon\,v^{-q}\,\left|\nabla v\right|^2\,\varphi^\theta+\varepsilon\,v^{-q}\,\varphi^\theta+\tfrac{1}{\varepsilon^{q-1}}\,|\nabla\varphi|^q\,\varphi^{\theta-q}\,,
					\end{aligned}
				\end{equation*}
				for all $\varepsilon>0$, where we have used the assumption that $q>2$ in the second inequality by the use of Young's inequality with exponents $(\frac{q}{2},\frac{q}{q-2})$. However, in the case where $q=2$, one observes that the second inequality is still trivially true.
				
				Let us plug this inequality back into \eqref{ibp-theta-ineq} to obtain
				\begin{equation}\label{pre-cutoff-ineq}
					\left(\tfrac{m}{2}+1-q-\theta\varepsilon\right)\!\int_{\mathcal{M}}e^{-f}v^{-q}\left|\nabla v\right|^2\varphi^\theta+\left(c_m-\theta\,\varepsilon\right)\int_{\mathcal{M}} e^{-f}\,v^{-q}\,\varphi^\theta
					\leq\frac{\theta}{\varepsilon^{q-1}}\int_{\mathcal{M}} e^{-f}\,|\nabla \varphi|^q\,\varphi^{\theta-q}\,.
				\end{equation}
				By taking $\theta>q$ and our assumption on $q$, we may now choose $\varepsilon>0$ small enough in terms of $m$ and $q$ so that both terms on the left-hand side of \eqref{pre-cutoff-ineq} are positive. Furthermore, by the properties \eqref{def-cutoff} of $\varphi=\varphi_R$, we have 
				\begin{equation*}
					\int_{\mathcal{B}_R(o)}e^{-f}\,v^{-q}\,\left|\nabla v\right|^2+\int_{\mathcal{B}_R(o)} e^{-f}\,v^{-q}\leq C\,R^{-q}\!\int_{\mathcal{A}_{R}(o)}e^{-f}\,,
				\end{equation*}
				where $C>0$ depends only on $d$, $m$, and $q$. The thesis trivially follows.
				
				\medskip
				\noindent{\textbf{Part ii):}}
				If $q=0$, the result follows by definition. When $2\leq q<\frac{m}{2}+1$, \eqref{int-est-2-zero} follows from \eqref{int-est-1-zero} of course. If $0<q<2$, simply by H\"{o}lder's inequality, we have
				\begin{equation*}
					\begin{aligned}
						\int_{\mathcal{B}_R(o)}e^{-f}\,v^{-q}&\leq \left(\int_{\mathcal{B}_R(o)}e^{-f}\,v^{-2}\right)^{\frac{q}{2}}\left(\int_{\mathcal{B}_R(o)}e^{-f}\right)^{\frac{2-q}{2}}\\
						&\leq C\,\left(\int_{\mathcal{B}_{2R}(o)}e^{-f}\right)^{\frac{q}{2}}\,R^{-q}\,\left(\int_{\mathcal{B}_R(o)}e^{-f}\right)^{\frac{2-q}{2}}\,,
					\end{aligned}
				\end{equation*}
				where we have used \eqref{int-est-1-zero}. The constant $C>0$ is as in the statement.
				
				Now it is only left to check that \eqref{int-est-2-zero} holds for $q=\frac{m}{2}+1$. We have, by applying \eqref{def-cutoff} to a particular case of \eqref{ibp-theta-ineq},
				\begin{equation}\label{ibp-critical}
					c_m\int_{\mathcal{B}_R(o)} e^{-f}\,v^{-\frac{m}{2}-1}
					\leq\frac{\theta}{R}\int_{\mathcal{A}_{R}(o)} e^{-f}\,v^{-\frac{m}{2}}\,|\nabla v|\,.
				\end{equation}
				Let us choose $\varepsilon\in(0,\frac{m-2}{2}\land 2)$. By H\"{o}lder's inequality followed by \eqref{int-est-2-zero} (with $0\leq q<\frac{m}{2}+1$) and \eqref{int-est-1-zero}, we have 
				\begin{equation*}
					\begin{aligned}
						\int_{\mathcal{A}_{R}(o)} e^{-f}\,v^{-\frac{m}{2}}\,|\nabla v|&\leq \left(\int_{\mathcal{A}_{R}(o)} e^{-f}\,v^{-\frac{m}{2}+1-\varepsilon}\right)^{\!\frac{1}{2}}\left(\int_{\mathcal{A}_{R}(o)} e^{-f}\,v^{-\frac{m}{2}-1+\varepsilon}\,|\nabla v|^2\right)^{\!\frac{1}{2}}\\
						&\leq C\,R^{-\frac{m}{2}}\,\mu(\mathcal{B}_{4R}(o))\,,
					\end{aligned}
				\end{equation*}
				which, combined with \eqref{ibp-critical} gives the thesis\footnote{In order to obtain the result with the ball of radius $2R$ on the right-hand side, it is only required to make some modifications of the test functions $\varphi$.}.
			\end{proof}
			
\begin{proof}[Proof of Theorem \ref{thm:vol-growth-liou}]
	Let us take $q=\frac{m}{2}+1$ in \eqref{int-est-2-zero}, and convert $v$ and $m$ back to the original variables $u$ and $p$. We obtain
	\begin{equation*}
		\int_{\mathcal{B}_{R}(o)}e^{-f}\,u^{p}\leq C\,\mu(\mathcal{B}_{2R}(o))\,R^{-\frac{2p}{p-1}}\,,
	\end{equation*}
	and the thesis follows from the hypothesis that $\mu(\mathcal{B}_{R}(o))=o\!\left(R^{\frac{2p}{p-1}}\right)$.
\end{proof}
\normalcolor
			
			The integral estimates in Lemma \ref{int-est-lemma} will most frequently be applied below in combination with a volume comparison result of the type \eqref{bis-gro}. Thus we state the following Corollary, whose proof is trivial.
			\begin{corollary}[Corollary to Lemma \ref{int-est-lemma}]\label{int-est-cor}
				Let ($\mathcal{M}$, $g$, $\mu$) be a complete weighted
				Riemannian manifold of dimension $d\geq2$ that admits a volume comparison theorem of the type \eqref{bis-gro} for some $n\geq d$, $o\in\mathcal{M}$, and all $r>0$.
				Let $v\in C^3(\mathcal{M})$ be a positive solution to \eqref{eq-for-v}. Then the following estimates hold for all  $R> r$:
				\begin{enumerate}[i)]
					\item If $2\leq q< \frac{m}{2}+1$, then
					\begin{equation}\label{int-est-1}
						\int_{\mathcal{B}_R(o)}e^{-f}\,v^{-q}\,\left|\nabla v\right|^2+\int_{\mathcal{B}_R(o)}e^{-f}\,v^{-q}\leq C\,R^{n-q}\,.
					\end{equation}
					\item If $0\leq q\leq \frac{m}{2}+1$\, then 
					\begin{equation}\label{int-est-2}
						\int_{\mathcal{B}_R(o)}e^{-f}\,v^{-q}\,\leq C\,R^{n-q}\,.
					\end{equation}
				\end{enumerate}
				The constant $C>0$ depends only on $d$, $n$, $m$, $q$, and the geometry of $\mathcal{M}$ in $\mathcal{B}_r(o)$.
			\end{corollary}
			
			\begin{proposition}\label{int-key-pro}
				Let $(\mathcal{M},g,\mu)$ be a complete weighted
				Riemannian manifold of dimension $d\geq2$.
				Let $v\in C^3(\mathcal{M})$ be a positive solution to \eqref{eq-for-v} that satisfies $W_{\!f}[v]\geq0$ (defined in \eqref{def-w}). 
				Then it holds 
				\begin{equation}\label{int-key-pro-eq}
					\left(\frac{t}{2}-\frac{1}{4}\right)^2\!\int_{\mathcal{M}} e^{-f}\,\mathtt{P}^{t-2}\,\left|\nabla \mathtt{P}\right|^2\,v^{2-m}\,\psi^\theta\leq{\theta^2}\!\int_{\mathcal{M}} e^{-f}\,\mathtt{P}^t\,v^{2-m}\,|\nabla\psi|^2\,\psi^{\theta-2}\,,
				\end{equation}
				for all $t\geq\frac{1}{2}$, $R>0$, $\theta\geq2$, and all $\psi\in C^{0,1}_c(\mathcal{M})$.
			\end{proposition}
			
			\begin{proof}
				Since $W_{\!f}[v]\geq0$ holds, \eqref{int-key-pro-eq} follows by a simple argument combining an integration by parts and H\"{o}lder's inequality. Indeed, let us multiply both sides of \eqref{fund-pw-ineq-eq} by $e^{-f}\,\varphi^\theta$.
				We may integrate by parts to obtain
				\begin{equation*}
					\left(t-\tfrac{1}{2}\right)\!\int_{\mathcal{M}} e^{-f}\,\mathtt{P}^{t-2}v^{2-m}|\nabla\mathtt{P}|^2\,\varphi^\theta\leq -\theta\!\int_{\mathcal{M}} e^{-f}\,\mathtt{P}^{t-1}\,v^{2-m}\,(\nabla\mathtt{P}\cdot\nabla\varphi)\,\varphi^{\theta-1}\,.
				\end{equation*}
				
				Now we apply H\"{o}lder's inequality on the right-hand side, rewritten in the form
				\begin{equation*}
					\begin{aligned}
						\int_{\mathcal{M}} \left(e^{-\frac{f}{2}}\mathtt{P}^\frac{t-2}{2}\,v^\frac{2-m}{2}\left|\nabla\mathtt{P}\right|\,\varphi^\frac{\theta}{2}\right)&\left(e^{-\frac{f}{2}}\,\mathtt{P}^\frac{t}{2}\,v^\frac{2-m}{2}\left|\nabla\varphi\right|\,\varphi^\frac{\theta-2}{2}\right)\\
						&\leq \left(\int_{\mathcal{M}} e^{-f}\mathtt{P}^{t-2}\,v^{2-m}\,\left|\nabla\mathtt{P}\right|^2\,\varphi^{\theta}\right)^{\frac{1}{2}}\left(\int_{\mathcal{M}} e^{-f}\,\mathtt{P}^t\,v^{2-m}\,\left|\nabla\varphi\right|^2\,\varphi^{\theta-2}\right)^{\frac{1}{2}}\,.
					\end{aligned}
				\end{equation*}
				We conclude simply by rearranging the terms.
			\end{proof}
			
			\subsection{Proof of Theorem \ref{thm-leq-crit}}
			Note that in this subsection, it holds $m=n$, since we are working in the critical case for equation \eqref{leq}. Also, since we assume that $\mathrm{Ric}_{n,d}\geq0$, \eqref{bis-gro} is in force, so Corollary \ref{int-est-cor} and in particular the estimates \eqref{int-est-1}--\eqref{int-est-2} hold.
			
			\begin{proof}
				The goal of the majority of the proof below will be dedicated to deriving that 
				\begin{equation}\label{p-const}
					\nabla\mathtt{P}=0\quad\mathrm{in}\;\;\mathcal{M}\,.
				\end{equation}
				
				\medskip
				\noindent{\textbf{Preliminary step:}}
				Let us show that \eqref{p-const} is enough to conclude the proof. By the maximum principle, either $u>0$ or $u=0$ everywhere; from now on, we will assume the positive case.
				First, it follows from \eqref{eq-for-v} and $\nabla\mathtt{P}\equiv0$ that $\mathcal{L}v=\mathtt{P}\equiv c$ for some constant $c$, which is \emph{positive} since $v>0$ in \eqref{eq-for-v}. On the other hand, it also follows from \eqref{ana-geo-eq} and $\nabla\mathtt{P}\equiv0$ that $\mathsf{k}[v]=0$. Now formula \eqref{eq-k} implies that the following statements all hold (recall that we are working under the assumption that $n=m$, so the second term in \eqref{eq-k} in any case vanishes):
				\begin{align}
					&\nabla^2v=\frac{\Delta v}{d}g\,,\label{imp-a}\\
					&d=n\quad\mathrm{or}\quad\frac{n}{n-d}\nabla f\cdot\nabla v=-\mathcal{L}v\equiv-c\,,\label{imp-c}\\
					&\mathrm{Ric}_{n,d}\left(\nabla v,\nabla v\right)=0\label{imp-d}\,.
				\end{align}
				It is a curious feature of the proof that \eqref{imp-d} is never used.
				Now, if the first case of \eqref{imp-c} holds, that is, if $d=n$, by the definition of \eqref{def-BER}, it follows that $\nabla f\equiv0$ so we reduce to the case of \cite[Theorem 1.1]{CFP} and there is nothing left to prove.
				
				Therefore to continue the chain of implications we assume from now on that $n> d$. Let us recall the definition of $\mathcal{L}$ and develop \eqref{imp-c}:
				\begin{equation*}
					-c=-\mathcal{L}v=-\Delta v+\nabla f\cdot\nabla v=-\Delta v-c\left(\frac{n-d}{n}\right)\,.
				\end{equation*}
				This implies that 
				\begin{equation*}
					\Delta v=\frac{cd}{n}
				\end{equation*}
				is in fact a constant. If we plug this expression back into \eqref{imp-a}, it turns out that 
				\begin{equation*}
					\nabla^2 v= \frac{c}{n}g\quad \mathrm{in}\;\;\mathcal{M}\,.
				\end{equation*}
				This guarantees by \cite[Theorem 2]{Tashiro} that $(\mathcal{M},g)$ is isometric to the Euclidean space $\R^d$, and so 
				$v$ is quadratic, \emph{i.e.}
				\begin{equation*}
					v(x)=\frac{c}{2n}|x-x_0|^2+b\,,
				\end{equation*}
				for some non-negative constant $b$ and some $x_0\in\mathcal{M}=\R^d$. By plugging this formula back into \eqref{eq-for-v} (and recalling that $m=n$), we find that $b=\frac{2}{n-2}\frac{1}{c}$, so it holds that
				\begin{equation}\label{weighted-bubble}
					v(x)=\frac{c}{2n}|x-x_0|^2+\frac{2}{n-2}\frac{1}{c}\,,
				\end{equation}
				for some $c>0$. This is precisely \eqref{at-bubble} with $d$ replaced by $n$, so if $n\neq d$, we \emph{would have found} a nontrivial \emph{weighted Euclidean Aubin-Talenti bubble}. 
				
				However, we can still derive a contradiction under the assumption that $n\neq d$, which we recall is equivalent to $\nabla f\neq0$. Indeed, let us observe that from \eqref{imp-c} and \eqref{weighted-bubble} it follows
				\begin{equation*}
					-c\left(\frac{n-d}{n}\right)=\nabla f\cdot\nabla v=\frac{c}{n}\,\nabla f\cdot(x-x_0)\,,
				\end{equation*}
				which implies that 
				\begin{equation*}
					\nabla f(x)=-(n-d)\frac{x-x_0}{|x-x_0|^2}\,.
				\end{equation*}
				Clearly, we now have a contradiction with the assumption $f\in C^2(\mathcal{M})$. 
				Therefore, we conclude that $\nabla f\equiv0$, $n=d$, and that the weighted bubble \eqref{weighted-bubble} is in fact the standard bubble \eqref{at-bubble} if $d\geq3$. If $d=2$ on the other hand, it follows from \cite{GS} that there are no non-trivial positive solutions, so $u=0$.
				
				From now on, we will prove that $\nabla\mathtt{P}\equiv0$.
				
				\medskip
				\noindent{\textbf{Proof of i), case $\boldsymbol{2<n<5}$.}} Let us begin by considering Proposition \ref{int-key-pro}, applying \eqref{int-key-pro-eq} with $t=\frac{1}{2}+\delta$, for some $\delta\in(0,\frac{1}{2})$ to be chosen later, and a fixed $\theta\geq2$ which will not play a role in this case. That is (after taking $\psi\mapsto\varphi_R$ as in \eqref{def-cutoff} for $R\geq1$),
				\begin{equation}\label{int-key-pro-eq-appl}
					\delta^2\!\int_{\mathcal{B}_R(o)} e^{-f}\,\mathtt{P}^{t-2}\,\left|\nabla \mathtt{P}\right|^2\,v^{2-n}\leq\frac{C}{R^2}\int_{\mathcal{A}_{R}(o)} e^{-f}\,\mathtt{P}^t\,v^{2-n}\,,
				\end{equation}
				for $C>0$ depending only on $d$, and
				where we have also chosen a general center point $o\in\mathcal{M}$.
				By definition, $1-t>0$ holds, and from \eqref{eq-for-v}, we have $1\leq c_n\, v\,\mathtt{P}$ for a suitable $c_n>0$ depending only on $n$. Therefore, it holds
				\begin{equation*}
					1\leq c_n(v\,\mathtt{P})^{1-t}\,,
				\end{equation*}
				which we apply to the right-hand side of \eqref{int-key-pro-eq-appl} (with $n=m$) to obtain
				\begin{equation}\label{int-key-pro-eq-appl-bis}
					\begin{aligned}
						\delta^2\!\int_{\mathcal{B}_R(o)} e^{-f}\,\mathtt{P}^{t-2}\,\left|\nabla \mathtt{P}\right|^2\,v^{2-n}&\leq\frac{C}{R^2}\int_{\mathcal{A}_{R}(o)} e^{-f}\,v^{2-n-t}\,(v\,\mathtt{P})\\
						&=\frac{C}{R^2}\int_{\mathcal{A}_{R}(o)} e^{-f}\,v^{2-n-t}\,(|\nabla v|^2+1)\,,
					\end{aligned}
				\end{equation}
				for a constant $C>0$ depending only on $d$ and $n$.
				Continuing the chain of inequalities, if
				\begin{equation}\label{n-cond-1}
					2\leq n-2+t<\frac{n}{2}+1\,
				\end{equation}
				holds (which is equivalent to $\frac{7}{2}\leq n<5$),
				the energy estimate \eqref{int-est-1} may be applied, yielding
				\begin{equation*}
					\int_{\mathcal{A}_{R}(o)} e^{-f}\,v^{2-n-t}\,(|\nabla v|^2+1)\leq C\, R^{2-t}\,,
				\end{equation*}
				for a constant $C>0$ depending only on $d$, $n$, $t$, and the geometry of $\mathcal{M}$ on $\mathcal{B}_1(o)$. Plugging this estimate back into \eqref{int-key-pro-eq-appl-bis} then yields
				\begin{equation}\label{pre-r-inf}
					\delta^2\!\int_{\mathcal{B}_R(o)} e^{-f}\,\mathtt{P}^{t-2}\,\left|\nabla \mathtt{P}\right|^2\,v^{2-n}\leq{C}\,R^{-t}\,,
				\end{equation}
				for a constant $C>0$ depending only on $d$, $n$, $t$, and the geometry of $\mathcal{M}$ on $\mathcal{B}_1(o)$,
				which, upon taking $R\to+\infty$, implies that $\nabla\mathtt{P}\equiv0$ as desired, for $n,t$ satisfying \eqref{n-cond-1}. 
				
				If, on the other hand, it holds
				\begin{equation}\label{n-cond-2}
					0\leq \frac{n-2-t}{1-t} \leq \frac{n}{2}+1\,,
				\end{equation}
				we conclude in a similar way. Note that, under our choice of $t$, \eqref{n-cond-2} is equivalent to $\frac{5}{2}<n<4$.
				Indeed, let us rewrite $\mathtt{P}^t\,v^{2-n}=(\mathtt{P}^t\,v^{-t})\,v^{2-n+t}$ and apply H\"{o}lder's inequality with exponents $\frac{1}{t}$ and $\frac{1}{1-t}$
				to the right-hand side of \eqref{int-key-pro-eq-appl}. We obtain
				\begin{equation}\label{holder-appl}
					\begin{aligned}
						\int_{\mathcal{A}_R} e^{-f}\,\mathtt{P}^t\,v^{2-n}&\leq\left(\int_{\mathcal{A}_R} e^{-f}\,\mathtt{P}\,v^{-1}\right)^{t}\left(\int_{\mathcal{A}_R} e^{-f}\,v^{\frac{2-n+t}{1-t}}\right)^{1-t}\\
						&\leq C\,\left(\int_{\mathcal{A}_R} e^{-f}\,v^{-2}\left(|\nabla v|^2+1\right)\right)^{t}\left(\int_{\mathcal{A}_R} e^{-f}\,v^{\frac{2-n+t}{1-t}}\right)^{1-t},
					\end{aligned}
				\end{equation}
				for a constant $C>0$ depending only on $n$, where in the second inequality we have used \eqref{eq-for-v}.
				
				Now, since \eqref{bis-gro} is in force, both integral terms on the right-hand side of \eqref{holder-appl} can be estimated using Corollary \ref{int-est-cor}: the first by \eqref{int-est-1}, since
				\begin{equation*}
					2<\frac{n}{2}+1\,,
				\end{equation*}
				and the second by \eqref{int-est-2}, since we recall that \eqref{n-cond-2} holds.
				In this way, from \eqref{holder-appl}, we obtain
				\begin{equation}\label{r1}
					\int_{\mathcal{A}_R} e^{-f}\,\mathtt{P}^t\,v^{2-n}\leq C\left(R^{n-2}\right)^{t}\left(R^{n-\frac{n-2-t}{1-t}}\right)^{1-t}=C\,R^{-t+2},
				\end{equation}
				for a $C>0$ depending only on $d$, $n$, $t$, and the geometry of $\mathcal{M}$ on $\mathcal{B}_1(o)$. Applying this to \eqref{int-key-pro-eq} yields \eqref{pre-r-inf} again, and we conclude by taking $R\to+\infty$ as before. Then $\nabla\mathtt{P}\equiv0$ for all $n,t$ satisfying \eqref{n-cond-1} or \eqref{n-cond-2}, the combination of which is equivalent to $\frac{5}{2}<n<5$, recalling the definition of $t$ and the arbitrariness of $\delta\in(0,\frac{1}{2})$.
				
				To handle the case $2<n<\frac{5}{2}$ we modify the previous argument, using additionally Lemma \ref{super-harm-lem} to make a pointwise estimate a positive power of $v$. Indeed, rewriting \eqref{super-harm-ineq} in terms of $v$, we have
				\begin{equation}\label{v-harm-appl}
					v(x)\leq C\,\mathrm{dist}(x,o)^2\quad\mathrm{in}\;\;\mathcal{M}\setminus \mathcal{B}_1(o)\,,
				\end{equation}
				for a constant $C>0$ depending only on $n$ and the behavior of $v$ in $\mathcal{B}_1(o)$. Similarly to above, we rewrite $\mathtt{P}^t\,v^{2-n}=v^{\frac{1}{2}}\,(\mathtt{P}^t\,v^{-t})\,v^{\frac{3}{2}-n+t}$, apply H\"{o}lder's inequality with exponents $\frac{1}{t}$ and $\frac{1}{1-t}$, and finally apply \eqref{v-harm-appl}, which yields
				\begin{equation}\label{holder-appl-bis}
					\begin{aligned}
						\int_{\mathcal{A}_R} e^{-f}\,\mathtt{P}^t\,v^{2-n}&\leq\norm{v}^{\frac{1}{2}}_{L^\infty(\mathcal{A}_R)}\left(\int_{\mathcal{A}_R} e^{-f}\,\mathtt{P}\,v^{-1}\right)^{t}\left(\int_{\mathcal{A}_R} e^{-f}\,v^{\frac{3/2-n+t}{1-t}}\right)^{1-t}\\
						&\leq C\,\norm{v}^{\frac{1}{2}}_{L^\infty(\mathcal{A}_R)}\left(\int_{\mathcal{A}_R} e^{-f}\,v^{-2}\left(|\nabla v|^2+1\right)\right)^{t}\left(\int_{\mathcal{A}_R} e^{-f}\,v^{\frac{3/2-n+t}{1-t}}\right)^{1-t}\\
						&\leq C\,R\,(R^{n-2})^t\left(R^{n-\frac{n-3/2-t}{1-t}}\right)^{1-t}\\
						&= C\,R^{-t+\frac{5}{2}}\,,
					\end{aligned}
				\end{equation}
				for a constant $C>0$ depending only on $d$, $n$, $t$, the behavior of $v$ in $\mathcal{B}_1(o)$, and the geometry of $\mathcal{M}$ on $\mathcal{B}_1(o)$. In \eqref{holder-appl-bis}, we have used the energy estimates contained Corollary \ref{int-est-cor} in the same way as before, this time assuming that 
				\begin{equation}\label{n-cond-3}
					0\leq \frac{n-\frac{3}{2}-t}{1-t} \leq \frac{n}{2}+1\,,
				\end{equation}
				which is equivalent to $2<n<\frac{10}{3}$. Plugging \eqref{holder-appl-bis} into \eqref{int-key-pro-eq-appl} yields 
				\begin{equation*}
					\int_{\mathcal{B}_R(o)} e^{-f}\,\mathtt{P}^{t-2}\,\left|\nabla \mathtt{P}\right|^2\,v^{2-n}\leq C\,R^{-t+\frac{1}{2}}=C\,R^{-\delta}\,,
				\end{equation*}
				where $\delta>0$ is chosen sufficiently small and now $C>0$ depends only on $d$, $n$, the behavior of $v$ in $\mathcal{B}_1(o)$, and the geometry of $\mathcal{M}$ on $\mathcal{B}_1(o)$. It now suffices to take $R\to+\infty$ to conclude that $\nabla\mathtt{P}\equiv0$. Note that the three conditions \eqref{n-cond-1}, \eqref{n-cond-2}, and \eqref{n-cond-3} taken together fully cover $2<n<5$.
				
				\medskip
				\noindent{\textbf{Proof of i), case $\boldsymbol{n=5}$.}} 
				As before, the starting point is the estimate \eqref{int-key-pro-eq}, but in this case $\theta\geq2$ is a variable to be chosen later. Let us carry all of the notation from the previous step, noting in particular that $R\geq1$. Applying \eqref{def-cutoff}, and setting $m=n=5$ and $t=\frac{1}{2}+\delta$, we have
				\begin{equation}\label{int-key-pro-eq-5}
					\delta^2\!\int_{\mathcal{M}} e^{-f}\,\mathtt{P}^{t-2}\,\left|\nabla \mathtt{P}\right|^2\,v^{-3}\,\varphi^\theta\leq C\,\frac{\theta^2}{R^2}\int_{\mathcal{A}_R} e^{-f}\,\mathtt{P}^t\,v^{-3}\,\varphi^{\theta-2}\,,
				\end{equation}
				for a $C>0$ depending only on $d$.
				We notice that, from \eqref{eq-for-v}, it follows 
				\begin{equation}\label{pineq0}
					1\le\frac{3}{2}\, v\,\mathtt{P} \,,
				\end{equation}
				\begin{equation}\label{pineq}
					|\nabla v|^2\le \frac{5}{2}\,v\,\mathtt{P} \,,
				\end{equation}
				and
				\begin{equation}\label{pineq2}
					\mathtt{P}\,v \leq \frac{5}{2\varepsilon}\left(\varepsilon\,|\nabla v|^2+\frac{2}{3}\right) \,,
				\end{equation}
				for any $\varepsilon \in (0,1]$.
				By raising \eqref{pineq0} to the power $t$, it follows
				\begin{equation} \label{comp23}
					\begin{aligned}
						\int_{\mathcal{A}_R}e^{-f}\,\mathtt{P}^t\,v^{-3}\,\varphi^{\theta-2}&\le \left(\frac{3}{2}\right)^{\frac12 + \delta}\int_{\mathcal{A}_R}e^{-f}\,\mathtt{P}^{1+2\delta}\,v^{-\frac{5}{2}+\delta}\,\varphi^{\theta-2}\,\\
						&\leq \frac{5}{2}\left(\frac{3}{2}\right)^{\frac12 + \delta}\int_{\mathcal{A}_R}e^{-f}\,\mathtt{P}^{2\delta}\,v^{-\frac{7}{2}+\delta}\left(|\nabla v|^2+\frac{2}{3}\right)\,\varphi^{\theta-2}\,.
					\end{aligned}
				\end{equation}
				where in the second inequality we have used \eqref{pineq2} with $\varepsilon=1$.
				We estimate the last term in \eqref{comp23} using Corollary \ref{ibp-cor}; take \eqref{ibp-formula-2} with $q=\frac{7}{2} -\delta$, $\ell= 2\delta$, and $\psi=\varphi^{\theta-2}$, obtaining
				\begin{equation}\label{comp14}
					\begin{aligned}
						C&\int_{\mathcal{A}_R}e^{-f}\,\mathtt{P}^{1+2\delta}\,v^{-\frac{5}{2}+\delta}\,\varphi^{\theta-2} \\ 
						&\le \underbrace{\int_{\mathcal{A}_R}e^{-f}\,\mathtt{P}^{2\delta}\,v^{-\frac{5}{2}+\delta} \left|\nabla v\cdot\nabla\varphi\right| \,\varphi^{\theta-3} }_{J_1}
						+\underbrace{\int_{\mathcal{A}_R}e^{-f}\,\mathtt{P}^{-1+2\delta}\,v^{-\frac{5}{2}+\delta} \left| \nabla v \cdot\nabla \mathtt{P}\right|\,\varphi^{\theta-2}}_{J_2}\,,
					\end{aligned}
				\end{equation}
				where $C>0$ only depends on $\theta$ and $\delta$.
				
				Now, let us focus on the first term on the right-hand side: $J_1$. An application of Young's inequality, with $\varepsilon_1>0$, $q>1$ and exponents $(q,\frac{q}{q-1})$ gives
				$$
				\mathtt{P}^{2\delta}\left| \nabla v\cdot\nabla\varphi \right|\,\varphi^{-1}  \leq \mathtt{P}^{1+ 2\delta}\,\varphi^{-1}\,\mathtt{P}^{-1}\left|\nabla  v|\, |\nabla \varphi \right|  \leq   C\, \mathtt{P}^{1+ 2\delta}  \left\{ \varepsilon_1^{1-q}\,\varphi^{-q}\,\mathtt{P}^{-q}\,\left| \nabla  v\right|^q\, \left|\nabla \varphi\right|^q  + \varepsilon_1   \right\} \,,
				$$
				for a constant $C$ depending only on $q$.
				Therefore, we have
				\begin{equation}\label{comp161}
					\begin{aligned}
						J_1&\le \underbrace{{C}\,{\varepsilon_1^{1-q}} \int_{\mathcal{A}_R}e^{-f}\,\mathtt{P}^{1-q+2\delta}\,v^{-\frac{5}{2}+\delta}\,\left|\nabla  v\right|^q\,\left|\nabla \varphi\right|^q\,\varphi^{\theta-2-q}}_{J_{1,A}} \\ 
						&+C \,\varepsilon_1 \int_{\mathcal{A}_R}e^{-f}\,\mathtt{P}^{1+2\delta}\,v^{-\frac{5}{2}+\delta}\,\varphi^{\theta-2}\,,
					\end{aligned}
				\end{equation}
				for $C>0$ depending only on $q$.
				From \eqref{pineq} and \eqref{def-cutoff} we obtain 
				$$
				J_{1,A}\le {C}\,{\varepsilon_1^{1-q}}\,R^{-q}\!\int_{\mathcal{A}_R}e^{-f}\,\mathtt{P}^{1-\frac{q}{2}+2\delta}\,v^{-\frac{5-q}{2}+\delta}\,\varphi^{\theta-2-q}\,,
				$$
				where $C>0$ depends only on $d$ and $q$.
				Now, by choosing $q=2(1+2\delta)$ and applying integral estimate \eqref{int-est-2}, we have
				$$
				J_{1,A}\le \frac{C}{\varepsilon_1^{1+4\delta}} R^{-2(1+2\delta)}\int_{\mathcal{A}_R}e^{-f}\,v^{-\frac{3}{2} + 3\delta}\,\varphi^{\theta-4-4\delta} \leq  \frac{C}{\varepsilon_1^{1+4\delta}}R^{ \frac32 - \delta } \,,
				$$ 
				for $C>0$ depending only on $d$, $n$, $\delta$, and the geometry of $\mathcal{M}$ on $\mathcal{B}_1(o)$.
				Thus from \eqref{comp161} we get 
				\begin{equation}\label{comp22}
					J_1\le \frac{C}{\varepsilon_1^{1+4\delta}}  R^{ \frac32 - \delta } +  C \,\varepsilon_1 \int_{\mathcal{A}_R}e^{-f}\,\mathtt{P}^{1+2\delta}\,v^{-\frac{5}{2}+\delta}\,\varphi^{\theta-2}\,,
				\end{equation}
				for $C>0$ depending only on $d$, $n$, $\delta$, and the geometry of $\mathcal{M}$ on $\mathcal{B}_1(o)$.
				
				Next, we consider $J_2$ in \eqref{comp14}. We apply Young's inequality to the integrand in the form
				\begin{equation*}
					\begin{aligned}
						\mathtt{P}^{-1+2\delta}\,v^{-\frac{5}{2}+\delta} \left| \nabla v \cdot\nabla \mathtt{P}\right|\,\varphi^{\theta-2}&\leq\mathtt{P}^{-1+2\delta}\,v^{-\frac{5}{2}+\delta} \left| \nabla v\right| \left|\nabla \mathtt{P}\right|\,\varphi^{\theta-2}\\
						&=\left(\frac{1}{\sqrt{2\,\varepsilon_0\,R^2}}\mathtt{P}^{-\frac{1}{4}+\frac{3\delta}{2}}\,v^{-1+\delta}\left|\nabla v\right|\,\varphi^{\frac{\theta}{2}-2}\right)\left(\sqrt{2\,\varepsilon_0\,R^2}\,\mathtt{P}^{-\frac{3}{4}+\frac{\delta}{2}}\,v^{-\frac{3}{2}}\left|\nabla\mathtt{P}\right|\,\varphi^{\frac{\theta}{2}}\right)\\
						&\leq\frac{1}{4\,\varepsilon_0\,R^2}\mathtt{P}^{-\frac{1}{2}+{3\delta}}\,v^{-2+2\delta}\left|\nabla v\right|^2\,\varphi^{\theta-4}+\varepsilon_0\,R^2\,\,\mathtt{P}^{-\frac{3}{2}+\delta}\,v^{-{3}}\left|\nabla\mathtt{P}\right|^2\,\varphi^{\theta}\,,
					\end{aligned}
				\end{equation*}
				for any $\varepsilon_0>0$.
				Plugging this back into $J_2$, we obtain
				\begin{equation}\label{comp20}
					J_2\le  \underbrace{\frac{1}{4\,\varepsilon_0\,R^2} \int_{\mathcal{A}_R}e^{-f}\,\mathtt{P}^{-\frac{1}{2}+3\delta}\,v^{-2+2\delta}\,\left|\nabla  v\right|^2\,\varphi^{\theta-4}}_{J_{2,A}} +\varepsilon_0\,R^{2}\!\int_{\mathcal{A}_R}e^{-f}\,\mathtt{P}^{-\frac{3}{2} + \delta}\,\left|\nabla \mathtt{P}\right|^2\,v^{-3}\,\varphi^{\theta}\,.
				\end{equation}
				Focusing on the first term, by applying \eqref{pineq} and then using H\"older's inequality we obtain
				\begin{equation}\label{comp33}
					\begin{aligned}
						J_{2, A} &\leq \frac{C}{\varepsilon_0\,R^2} \int_{\mathcal{A}_R}e^{-f}\,\mathtt{P}^{\frac{1}{2}+3\delta}\,v^{-1+2\delta}\,\varphi^{\theta-4}\\
						&\leq \frac{C}{\varepsilon_0\,R^2} \left(\int_{\mathcal{A}_R}e^{-f}\,\mathtt{P}^{(\frac{1}{2}+3\delta)\frac{3}{2}}\,v^{(-1+3\delta)\frac{3}{2}}\,\varphi^{\theta-4}\right)^\frac{2}{3}\left(\int_{\mathcal{A}_R}e^{-f}\,v^{-3\delta}\,\varphi^{\theta-4}\right)^\frac{ 1}{3} \,,   
					\end{aligned}
				\end{equation}
				for some numerical $C>0$.
				Now from \eqref{pineq0} we have that
				\begin{equation*}
					1\le \left(\frac{3}{2}\,\mathtt{P}\,v\right)^{1-(\frac{1}{2}+3\delta)\frac{3}{2}}\,,
				\end{equation*}
				for a sufficiently small $\delta$,
				which, inserted into \eqref{comp33}, yields
				\begin{equation}\label{comp34}
					J_{2,A} \leq \frac{C}{\varepsilon_0\,R^2} \left( \int_{\mathcal{A}_R}e^{-f}\,v^{-\frac{9}{4}}\,(\mathtt{P}\,v)\,\varphi^{\theta-4}\right)^\frac{2}{3}\left(\int_{\mathcal{A}_R}e^{-f}\,v^{-3\delta}\,\varphi^{\theta-4}\right)^\frac{ 1}{3},
				\end{equation}
				for a numerical $C>0$.
				Since $n=5$, it is easy to check that the integral estimates \eqref{int-est-1}--\eqref{int-est-2} may be applied to the first and second terms, respectively, of \eqref{comp34} (first applying, as usual \eqref{pineq2} to the first term). In this way, we obtain
				\begin{equation}\label{comp35}
					J_{2,A}\le \frac{C}{\varepsilon_0}  R^{\frac{3}{2}-\delta}\,,
				\end{equation}
				for $C>0$ depending only on $d$, $n$, and the geometry of $\mathcal{M}$ in $\mathcal{B}_1(o)$.
				Then from \eqref{comp20} and \eqref{comp35}, we obtain that
				\begin{equation}\label{comp21}
					J_2\le \frac{C}{\varepsilon_0}  R^{\frac{3}{2}-\delta} + \varepsilon_0\,R^{2}\!\int_{\mathcal{A}_R}e^{-f}\,\mathtt{P}^{-\frac{3}{2} + \delta}\,\left|\nabla \mathtt{P}\right|^2\,v^{-3}\,\varphi^{\theta}\,,
				\end{equation}
				for all $\varepsilon_0>0$, and for $C>0$ depending only on $d$, $n$, and the geometry of $\mathcal{M}$ in $\mathcal{B}_1(o)$.
				Putting together \eqref{comp14}, \eqref{comp22} and \eqref{comp21} we have 
				\begin{equation*}
					\begin{aligned}
						C\int_{\mathcal{A}_R}e^{-f}\,\mathtt{P}^{1+2\delta}\,v^{-\frac{5}{2}+\delta}\,\varphi^{\theta-2}&\le  \left(\varepsilon_1^{-(1+4\delta)}+ \varepsilon_0^{-1} \right)  R^{ \frac32 - \delta }\\
						&+  \varepsilon_1 \int_{\mathcal{A}_R}e^{-f}\,\mathtt{P}^{1+2\delta}\,v^{-\frac{5}{2}+\delta}\,\varphi^{\theta-2} \\
						&+ \varepsilon_0\,R^{2}\!\int_{\mathcal{A}_R}e^{-f}\,\mathtt{P}^{-\frac{3}{2} + \delta}\,\left|\nabla \mathtt{P}\right|^2\,v^{-3}\,\varphi^{\theta}\,,
					\end{aligned}
				\end{equation*}
				for any $\varepsilon_0>0$, $\varepsilon_1>0$, $R\geq1$, and a constant $C>0$ depending only on $d$, $n$, $\theta$, $\delta$, and the geometry of $\mathcal{M}$ in $\mathcal{B}_1(o)$. From now on in this proof $C>0$ will have only these dependencies.
				In particular, for $\varepsilon_1$ small enough (depending on the above quantities through the constant $C$), the second term on the right-hand side may be absorbed into the left-hand side, yielding 
				\begin{equation}\label{comp36}
					C\int_{\mathcal{A}_R}e^{-f}\,\mathtt{P}^{1+2\delta}\,v^{-\frac{5}{2}+\delta}\,\varphi^{\theta-2} \le \varepsilon_0^{-1}\,R^{\frac{3}{2}-\delta}+\varepsilon_0\,R^{2}\!\int_{\mathcal{A}_R}e^{-f}\,\mathtt{P}^{-\frac{3}{2} + \delta}\,\left|\nabla \mathtt{P}\right|^2\,v^{-3}\,\varphi^{\theta}\,,
				\end{equation}
				where $C>0$ is as before, and we must have $\theta\geq4$, and $\delta\in(0,\frac{1}{2})$, the latter of which must also be sufficiently small, due to the step before \eqref{comp34}. 
				
				To conclude, we combine \eqref{int-key-pro-eq-5}, \eqref{comp23}, and \eqref{comp36}, which finally yields
				\begin{equation*}
					\int_{\mathcal{M}} e^{-f}\,\mathtt{P}^{t-2}\,\left|\nabla \mathtt{P}\right|^2\,v^{-3}\,\varphi^\theta\leq C\,R^{-2}\left(\varepsilon_0^{-1}\,R^{\frac{3}{2}-\delta}+\varepsilon_0\,R^{2}\!\int_{\mathcal{A}_R}e^{-f}\,\mathtt{P}^{-\frac{3}{2} + \delta}\left|\nabla \mathtt{P}\right|^2\,v^{-3}\,\varphi^{\theta}\right)\,,
				\end{equation*}
				where $C>0$ is as before. Recalling that $t-2=-\frac{3}{2}+\delta$, by choosing a sufficiently small $\varepsilon_0$ (in terms of $\theta$, $\delta$), we may absorb the integral on the right-hand side into the left-hand side, which gives
				\begin{equation}\label{r2}
					\int_{\mathcal{B}_R(o)} e^{-f}\,\mathtt{P}^{t-2}\,\left|\nabla \mathtt{P}\right|^2\,v^{-3}\leq C\,R^{-\frac{1}{2}-\delta}\,,\quad\mathrm{for}\;\mathrm{all}\;\;R\geq1\,,
				\end{equation}
				and for a $C>0$ depending only on $d$, $n$, $\theta$, $\delta$, and the geometry of $\mathcal{M}$ in $\mathcal{B}_1(o)$. Taking $R\to+\infty$ again yields $\nabla\mathtt{P}\equiv0$.
				
				\medskip
				\noindent{\textbf{Proof of ii)}}
				Let us fix a general $o\in\mathcal{M}$ and $R\geq1$ as usual.
				Let us consider \eqref{int-key-pro-eq} with $t>\frac{1}{2}$ to be chosen later,
				and $\theta\geq2$ fixed. That is,
				\begin{equation}\label{40}
					\begin{aligned}
						\int_{\mathcal{B}_R(o)} e^{-f}\,\mathtt{P}^{t-2}\,\left|\nabla \mathtt{P}\right|^2\,v^{2-n}\leq \frac{C}{R^2}\!\int_{\mathcal{A}_R(o)} e^{-f}\,\mathtt{P}^t\,v^{2-n}\,,
					\end{aligned}
				\end{equation}
				for $C>0$ depending only on $d$ and $t$.
				Let us estimate the right-hand side, using \eqref{eq-for-v}:
				\begin{equation*}
					\mathtt{P}^t\,v^{2-n}=(v\,\mathtt{P})^t\,v^{2-n-t}\leq C_1\!\left(|\nabla v|^{2t}+C_2\right)\,v^{2-n-t}\,,
				\end{equation*}
				for $C_1$ and $C_2$ positive and depending at most on $n$ and $t$. This yields
				\begin{equation}\label{41}
					\frac{C}{R^2}\!\int_{\mathcal{A}_R(o)} e^{-f}\,\mathtt{P}^t\,v^{2-n}\leq \underbrace{\frac{1}{R^2}\!\int_{\mathcal{A}_R(o)} e^{-f}\,|\nabla v|^{2t}\,v^{2-n-t}}_{\mathsf{I}_R}+ \underbrace{\frac{1}{R^2}\!\int_{\mathcal{A}_R(o)} e^{-f}\,v^{2-n-t}}_{\mathsf{II}_R}\,,
				\end{equation}
				for $C>0$ depending only on $d$ and $t$.
				
				In order to continue the estimate of $\mathsf{I}_R$, it is useful to observe that by the assumption \eqref{pw-decay-cond} it follows that 
				\begin{equation}\label{pw-decay-cond-v}
					v^{-1}(x)\leq C\left(\mathrm{dist}(x,o)\right)^{-\frac{2\alpha}{n-2}}\quad\mathrm{in}\;\;\mathcal{M}\setminus \mathcal{B}_1(o)\,,
				\end{equation}
				where $C>0$ depends only on the behavior of $v$ in $\mathcal{B}_1(o)$.
				Furthermore, from the Yau-type estimate Lemma \ref{yau-lem} and \eqref{pw-decay-cond-v}, it can be derived (following \cite[Corollary 2.2]{FMM}) that
				\begin{equation}\label{yau-appl}
					\frac{|\nabla v(x)|^2}{v(x)^2}\leq C \left(\mathrm{dist}(x,o)\right)^{-\frac{4\alpha}{n-2}}\quad\mathrm{in}\;\;\mathcal{M}\setminus \mathcal{B}_1(o)\,,
				\end{equation}
				where $C>0$ only depends on $d$, $n$, and the behavior of $v$ in $\mathcal{B}_1(o)$.
				With \eqref{pw-decay-cond-v}--\eqref{yau-appl} in hand and with the choice $t=\frac{1}{2}+\varepsilon$ for a small $\varepsilon>0$, we may estimate $\mathsf{I}_R$ in the following way
				\begin{equation}\label{42}
					\begin{aligned}
						\mathsf{I}_R=R^{-2}\!\int_{\mathcal{A}_R(o)} e^{-f}\,|\nabla v|^{1+2\varepsilon}\,v^{\frac{3}{2}-n-\varepsilon}&=R^{-2}\!\int_{\mathcal{A}_R(o)} e^{-f}\,\left(\frac{|\nabla v|}{v}\right)^{\!2\varepsilon}\,v^{\frac{-(n-5-2\varepsilon)}{2}}\,|\nabla v|\,v^{-\frac{n}{2}-1+{\varepsilon}}\,\\
						&\leq C\,R^{-2}\,R^{-\frac{4\alpha\varepsilon}{n-2}}\,R^{-\frac{\alpha}{n-2}(n-5-2\varepsilon)}\!\int_{\mathcal{A}_R(o)} e^{-f}\,|\nabla v|\,v^{-\frac{n}{2}-1+{\varepsilon}}\,,
					\end{aligned}
				\end{equation}
				for a $C>0$ depending on $d$, $n$, and the behavior of $v$ in $\mathcal{B}_1(o)$. In the application of \eqref{pw-decay-cond-v}, we have used the fact that $n>5$, and implicitly made the additional restriction $\varepsilon<\frac{n-5}{2}$. We now complete the argument by applying H\"{o}lder's inequality and integral estimates \eqref{int-est-1}--\eqref{int-est-2} to the right-hand side of \eqref{42}. That is,
				\begin{equation*}
					\begin{aligned}
						\int_{\mathcal{A}_R(o)} e^{-f}\,|\nabla v|\,v^{-\frac{n}{2}-1+{\varepsilon}}&\leq \left(\int_{\mathcal{A}_R(o)} e^{-f}\,v^{-\frac{n}{2}-1+{\varepsilon}}\,|\nabla v|^2\right)^{\frac{1}{2}}\left(\int_{\mathcal{A}_R(o)} e^{-f}\,v^{-\frac{n}{2}-1+{\varepsilon}}\right)^{\frac{1}{2}}\\
						&\leq C\,R^{\frac{n}{2}-1+{\varepsilon}}\,,
					\end{aligned}
				\end{equation*}
				for $C>0$ depending only on $d$, $n$, $\varepsilon$, and the geometry of $\mathcal{M}$ in $\mathcal{B}_1(o)$. Plugging this estimate back into \eqref{42} yields
				\begin{equation*}
					\mathsf{I}_R\leq C\,R^{\frac{-2\alpha(n-5+2\varepsilon)+(n-2)(n-6+2\varepsilon)}{2(n-2)}}\,,
				\end{equation*}
				for a $C>0$ depending only on $d$, $n$, $\varepsilon$, the geometry of $\mathcal{M}$ in $\mathcal{B}_1(o)$ and the behavior of $v$ in $\mathcal{B}_1(o)$. In view of \eqref{pw-decay-cond-1}, we may take $\varepsilon>0$ sufficiently small so that 
				\begin{equation}\label{ito01}
					\lim_{R\to+\infty}\mathsf{I}_R=0\,.
				\end{equation}
				
				Let us now estimate $\mathsf{II}_R$, using only \eqref{pw-decay-cond-v} and \eqref{int-est-2}. We obtain
				\begin{equation*}
					\begin{aligned}
						\mathsf{II}_R=R^{-2}\!\int_{\mathcal{A}_R(o)} e^{-f}\,v^{\frac{5-n-2\varepsilon}{2}}\,v^{-\frac{n}{2}-1}\leq C\,R^{-2}\,R^{-\frac{2\alpha}{n-2}\frac{n-5+2\varepsilon}{2}}\,R^{\frac{n}{2}-1}=C\,R^{\frac{-2\alpha(n-5+2\varepsilon)+(n-2)(n-6)}{2(n-2)}}\,,
					\end{aligned}
				\end{equation*}
				for a $C>0$ depending only on $d$, $n$, $\varepsilon$, the geometry of $\mathcal{M}$ in $\mathcal{B}_1(o)$ and the behavior of $v$ in $\mathcal{B}_1(o)$. Applying once more \eqref{pw-decay-cond-1}, we derive that
				\begin{equation}\label{iito01}
					\lim_{R\to+\infty}\mathsf{II}_R=0\,.
				\end{equation}
				In view of \eqref{40}, \eqref{41}, \eqref{ito01}, and  \eqref{iito01}, it follows that $\nabla\mathtt{P}\equiv0$.
				
				\medskip
				\noindent{\textbf{Proof of iii)}} 
				Let us begin by considering \eqref{41}, now with $t=\frac{n-2}{2}$. As before, we use \eqref{pw-decay-cond-v} and the Yau-type estimate \eqref{yau-appl}. We derive
				\begin{equation*}
					\begin{aligned}
						\mathsf{I}_R&=\frac{1}{R^2}\int_{\mathcal{A}_R(o)} e^{-f}\,v^{-\varepsilon}\,\left(\frac{|\nabla v|}{v}\right)^{\!n-4}\,v^{-\frac{n}{2}-1+\varepsilon}\,|\nabla v|^2\\
						&\leq C\,R^{-2}\,R^{-\frac{2\alpha\varepsilon}{n-2}}\,R^{-\frac{2\alpha(n-4)}{n-2}}\int_{\mathcal{A}_R(o)} e^{-f}\,v^{-\frac{n}{2}-1+\varepsilon}\,|\nabla v|^2\\
						&\leq C\,R^{-2}\,R^{-\frac{2\alpha\varepsilon}{n-2}}\,R^{-\frac{2\alpha(n-4)}{n-2}}\,R^{\frac{n}{2}-1+\varepsilon}\\
						&= C\,R^{\frac{-4\alpha(n-4+\varepsilon)+(n-6+2\varepsilon)(n-2)}{2(n-2)}}\,,
					\end{aligned}
				\end{equation*}
				for $C>0$ depending only on $d$, $n$, $\varepsilon$, the geometry of $\mathcal{M}$ on $\mathcal{B}_1(o)$, and the behavior of $v$ in $\mathcal{B}_1(o)$, where in the last step we have used integral estimate \eqref{int-est-1}. Clearly, in view of \eqref{pw-decay-cond}, it follows, taking $\varepsilon>0$ suitably small, that 
				\begin{equation}\label{conv1}
					\mathsf{I}_R\to0\,,\quad\mathrm{as}\;\;R\to+\infty\,.
				\end{equation}
				
				It is more straightforward to estimate $\mathsf{II}_R$. We find, using \eqref{pw-decay-cond} and \eqref{int-est-2}
				\begin{equation*}
					\begin{aligned}
						\mathsf{II}_R=R^{-2}\!\int_{\mathcal{A}_R(o)} e^{-f}\,v^{-n+4}\,v^{-\frac{n}{2}-1}\leq C\,R^{-2}\,R^{-\frac{2\alpha}{n-2}(n-4)}\,R^{\frac{n}{2}-1}=C\,R^{\frac{-4\alpha(n-4)+(n-2)(n-6)}{2(n-2)}}\,,
					\end{aligned}
				\end{equation*}
				for $C>0$ depending only on $d$, $n$, the geometry of $\mathcal{M}$ on $\mathcal{B}_1(o)$, and the behavior of $v$ in $\mathcal{B}_1(o)$. Again, by \eqref{pw-decay-cond}, it follows that 
				\begin{equation*}
					\mathsf{II}_R\to0\,,\quad\mathrm{as}\;\;R\to+\infty\,,
				\end{equation*}
				which, combined with \eqref{40}, \eqref{41}, and \eqref{conv1} yields $\nabla\mathtt{P}\equiv0$.
				
				\medskip
				\noindent{\textbf{Proof of iv)}}
				We first claim that 
				\begin{equation}\label{global-energy-claim}
					\int_{\mathcal{M}}e^{-f}\,u^{\frac{2n}{n-2}}<+\infty\quad\mathrm{implies}\;\;\int_{\mathcal{M}}e^{-f}\,\left|\nabla u\right|^2<+\infty\,,
				\end{equation}
				the proof of which is quite standard. To this end, let us fix a general $o\in\mathcal{M}$ and multiply \eqref{leq-crit} by $e^{-f}\,u\,\varphi_R^2$, where $\varphi_R$ is the standard cutoff function defined by Lemma \ref{cutoff-lem} with $R\geq1$, and integrate by parts. We obtain, using Young's inequality
				\begin{equation*}
					\begin{aligned}
						\int_{\mathcal{M}}e^{-f}\,\left|\nabla u\right|^2\,\varphi_R^2&\leq\int_{\mathcal{M}}e^{-f}\,u^{\frac{2n}{n-2}}\,\varphi_R^2+2\int_{\mathcal{M}}e^{-f}\,u\,\varphi_R\,|\nabla u|\,|\nabla\varphi_R|\\
						&\leq\int_{\mathcal{M}}e^{-f}\,u^{\frac{2n}{n-2}}\,\varphi_R^2+\frac{1}{2}\int_{\mathcal{M}}e^{-f}\,\left|\nabla u\right|^2\,\varphi_R^2+2\int_{\mathcal{M}}e^{-f}\,u^2\,|\nabla\varphi_R|^2\,.
					\end{aligned}
				\end{equation*}
				Next, applying H\"{o}lder's inequality and \eqref{def-cutoff} yields
				\begin{equation*}
					\begin{aligned}
						\frac{1}{2}\int_{\mathcal{B}_R(o)}e^{-f}\,\left|\nabla u\right|^2&\leq\int_{\mathcal{M}}e^{-f}\,u^{\frac{2n}{n-2}}+2\int_{\mathcal{M}}e^{-f}\,u^2\,|\nabla\varphi_R|^2\\
						&\leq\int_{\mathcal{M}}e^{-f}\,u^{\frac{2n}{n-2}}+2\,\left(\int_{\mathcal{M}}e^{-f}\,u^{\frac{2n}{n-2}}\right)^{\!\frac{n-2}{2}}\left(\int_{\mathcal{M}}e^{-f}\,|\nabla\varphi_R|^n\right)^{\!\frac{2}{n}}\\
						&\leq \int_{\mathcal{M}}e^{-f}\,u^{\frac{2n}{n-2}}+C\left(\int_{\mathcal{M}}e^{-f}\,u^{\frac{2n}{n-2}}\right)^{\!\frac{n-2}{2}},
					\end{aligned}
				\end{equation*}
				for all $R\geq1$, where in the last passage we have used \eqref{bis-gro}, and for a constant $C>0$ depending only on $d$, $n$, and the geometry of $\mathcal{M}$ in $\mathcal{B}_1(o)$.
				We take $R\to+\infty$ to complete the proof of the claim \eqref{global-energy-claim}.
				
				Now to prove $\nabla\mathtt{P}\equiv0$, we fix a general $o\in\mathcal{M}$, and consider \eqref{fund-pw-ineq-eq}--\eqref{def-w} with $t=1$ and $m=n$. It holds $W_{\!f}[v]\geq0$ as usual so we may ignore the second term on the left-hand side of \eqref{fund-pw-ineq-eq}. We now multiply \eqref{fund-pw-ineq-eq} by $e^{-f}\,\varphi_R^\theta$ and integrate by parts, where $\theta\geq2$ is fixed and $R\geq1$. It holds, also using H\"{o}lder's inequality, that
				\begin{equation}\label{4}
					\begin{aligned}
						\frac{1}{2}\int_{\mathcal{M}}e^{-f}\,\mathtt{P}^{-1}\,v^{2-n}\,|\nabla\mathtt{P}|^2\,\varphi_R^\theta&\leq -\theta\int_{\mathcal{M}}e^{-f}\,v^{2-n}\,\varphi_R^{\theta-1}\,(\nabla\mathtt{P}\cdot\nabla\varphi_R)\\
						&\leq \theta\left(\int_{\mathcal{M}} e^{-f}\,\mathtt{P}^{-1}\,v^{2-n}\,|\nabla\mathtt{P}|^2\,\varphi_R^\theta\right)^{\frac{1}{2}}\left(\int_\mathcal{M}e^{-f}\,\mathtt{P}\,v^{2-n}\,\varphi_R^{\theta-2}\,|\nabla\varphi_R|^2\right)^{\frac{1}{2}}\,.
					\end{aligned}
				\end{equation}
				The first term on the right-hand side of \eqref{4} is absorbed into its left-hand side.
				
				Next, from \eqref{eq-for-v} and making the change of variables $v=u^{-\frac{2}{n-2}}$, the following estimate on the second term may be easily checked: 
				\begin{equation}\label{5}
					\begin{aligned}
						\int_\mathcal{M}e^{-f}\,\mathtt{P}\,v^{2-n}\,\varphi_R^{\theta-2}\,|\nabla\varphi_R|^2&\leq  \frac{C}{R^2}\int_{\mathcal{A}_R}e^{-f}\,(\mathtt{P}\,v)\,v^{1-n}\\
						&=\frac{C}{R^2}\int_{\mathcal{A}_R}e^{-f}\,v\left(\frac{2n}{(n-2)^2}\left|\nabla u\right|^2+\frac{2}{n-2}\,u^{\frac{2n}{n-2}}\right),
					\end{aligned}
				\end{equation}
				where $C>0$ depends only on $d$.
				
				Now, by Lemma \ref{super-harm-lem}, applied to $v$ for $R=1$, it follows that
				\begin{equation}\label{super-harm-appl}
					v(x)\leq C\,{\mathrm{dist}(x,o)^{2}}\quad\mathrm{if}\;\;\mathrm{dist}(x,o)\geq1\,,
				\end{equation}
				where $C>0$ depends only on $n$ and the behavior of $v$ on $\mathcal{B}_1(o)$.
				If we insert \eqref{super-harm-appl} into \eqref{5}, we obtain
				\begin{equation}\label{6}
					\begin{aligned}
						\int_\mathcal{M}e^{-f}\,\mathtt{P}\,v^{2-n}\,\varphi_R^{\theta-2}\,|\nabla\varphi_R|^2\leq C\!\int_{\mathcal{A}_R}e^{-f}\left(\frac{2n}{(n-2)^2}\left|\nabla u\right|^2+\frac{2}{n-2}\,u^{\frac{2n}{n-2}}\right),
					\end{aligned}
				\end{equation}
				where $C>0$ depends only on $d$, $n$, and the behavior of $u$ on $\mathcal{B}_1(o)$. Inserting \eqref{6} back into \eqref{4} and suitably rearranging terms brings us to 
				\begin{equation*}
					\int_{\mathcal{B}_R(o)}e^{-f}\,\mathtt{P}^{-1}\,v^{2-n}\,|\nabla\mathtt{P}|^2\leq C\!\int_{\mathcal{A}_R(o)}e^{-f}\left(\left|\nabla u\right|^2+u^{\frac{2n}{n-2}}\right),
				\end{equation*}
				for a $C>0$ depending only on $d$, $n$, and the behavior of $u$ on $\mathcal{B}_1(o)$. In view of the hypothesis and claim \eqref{global-energy-claim}, we may take $R\to+\infty$ to determine that $\nabla\mathtt{P}\equiv0$.
			\end{proof}
			
			\section{On infinity Bakry-\'{E}mery Ricci curvature}\label{sec:infty}
			We may make some adaptations to the strategy employed in the previous sections to prove our results concerning $\infty$-Bakry-\'{E}mery Ricci curvature -- Theorem \ref{thm-infty-crit}, Theorem \ref{thm:exp-class}, and Theorem \ref{thm-infty}. First of all, since $n=\infty$, all references to $n$ disappear from the considerations, and we are just left with $d$, \emph{i.e.} the (integer) dimension of $\mathcal{M}$, and $m$, the sub-criticality index defined in \eqref{def-m} (when we are considering positive solutions to \eqref{leq}). On the other hand, for solutions of \eqref{liou-eq},
			we have $d=m=2$.

			The basic technical impediment for obtaining non-existence and classification/rigidity results involving $\mathrm{Ric}_{\infty,d}$ curvature is that $W_{\!f}[v]$, a term appearing on the \emph{left-hand side} of \eqref{fund-pw-ineq-eq} reduces in this context to
			\begin{equation}\label{def-w-inf}
				\begin{aligned}
					W_{\!f}[v]&=\tfrac{m-d}{m^2}(\mathcal{L}v)^2+\tfrac{2}{m}\mathcal{L}v\,\nabla f\cdot\nabla v+\mathrm{Ric}_{\infty,d}(\nabla v,\nabla v)\,,
				\end{aligned}
			\end{equation}
			which \emph{may be negative} if $\nabla f\cdot\nabla v$ is. However, thanks to the assumption \eqref{dot-cond}, we ensure that in fact $\nabla f\cdot\nabla v\geq0$ (recall that $\nabla v$ and $\nabla u$ have opposite signs when $u=v^{-\frac{2}{p-1}}$ or when $u=-2\ln v$).
			
			Next, when one attempts to prove Theorem \ref{thm-infty-crit} or Theorem \ref{thm:exp-class} via an adaptation of the proof of Theorem \ref{thm-leq-crit}, it turns out that 
			$\mathtt{P}\equiv c$ does not imply in a straightforward way that $\mathcal{M}$ is diffeomorphic to $\R^d$ as in the preliminary step of the proof of Theorem \ref{thm-leq-crit}. Indeed \eqref{ana-geo-eq} still holds, so $\mathtt{P}\equiv c$ still implies that $k[v]=0$, but \eqref{eq-k} reduces to
			\begin{equation}\label{eq-k-inf}
				\mathsf{k}[v]=\norm{\nabla^2v-\frac{\Delta v}{d}g\,}^2_{\mathrm{H.S.}}\!+\frac{2}{d}\,\mathcal{L}v\,\nabla f\cdot\nabla v+\frac{1}{d}\left|\nabla f\cdot\nabla v\right|^2+\mathrm{Ric}_{\infty,d}\left(\nabla v,\nabla v\right),
			\end{equation}
			where we have also used the fact that we are in the critical case $m=d$. Again, it is only due to our assumption \eqref{dot-cond} that we may deduce useful information from \eqref{eq-k-inf} when $k[v]=0$.
			
			\subsection{Proof of Theorem \ref{thm-infty}}
			\begin{proof}
				By the maximum principle for $\mathcal{L}$, either $u=0$ everywhere or $u>0$ everywhere. In order to obtain a contradiction, let us assume the latter case, so it follows that $v>0$ and $\mathtt{P}>0$. {Also, from now on we assume that $d\geq3$, since $d=2$ and the assumption \eqref{bis-gro-inf} implies that $\mu(\mathcal{B}_R)=\mathcal{O}(R^2)$, so Theorem \ref{thm:vol-growth-liou} holds for all $p>1$.}
				
				Keeping in mind the stated non-negativity of $W_{\!f}[v]$, let us adapt the proof of Theorem \ref{thm-subcrit}, making suitable changes where necessary. First, \eqref{fund-pw-ineq-eq} still holds, but now with the modified definition of $W_{\!f}[v]$ given in \eqref{def-w-inf}. Then we may multiply all terms in \eqref{fund-pw-ineq-eq} by $e^{-f}\,\varphi_R^\theta$ (taking $\varphi_R=\varphi$ for simplicity) and integrate by parts as before. In fact, we may continue the proof (recalling that $\mathrm{Ric}_{\infty,d}\geq0$) up to \eqref{young-1}, which in its modified form reads
				\begin{equation}\label{young-1-inf}
					\begin{aligned}
						\left(t-\tfrac{1}{2}-\varepsilon\right)\!\int_{\mathcal{M}} e^{-f}\,\mathtt{P}^{t-2}\,v^{2-m}\,&|\nabla\mathtt{P}|^2\,\varphi^\theta+\frac{m-d}{m}\int_{\mathcal{M}} e^{-f}\,\mathtt{P}^{t+1}\,v^{1-m}\,\varphi^\theta\\&+2\int_{\mathcal{M}} e^{-f}\,\mathtt{P}^{t}\,v^{1-m}\,(\nabla f\cdot\nabla v)\,\varphi^\theta\leq \frac{\theta^2}{\varepsilon}\!\int_{\mathcal{M}} e^{-f}\,\mathtt{P}^{t}\,v^{2-m}\,|\nabla \varphi|^2\,\varphi^{\theta-2}\,.    
					\end{aligned}
				\end{equation}
				We apply Young's inequality to the right-hand side as
				\begin{equation*}
					\tfrac{\theta^2}{\varepsilon}\mathtt{P}^{t}\,v^{2-m}\,|\nabla \varphi|^2\,\varphi^{\theta-2}\leq{\varepsilon}\,\mathtt{P}^{t+1}\,v^{1-m}\,\varphi^{\theta}+\mathtt{K}\,v^{2-m+t}\,|\nabla\varphi|^{2(t+1)}\,\varphi^{\theta-2(t+1)}\,,
				\end{equation*}
				for $\mathtt{K}>0$ depending only on $\theta$ and $\varepsilon$.
				This yields (discarding the third term on the left-hand side of \eqref{young-1-inf}, which by assumption is non-negative)
				\begin{equation}\label{quasi}
					\begin{aligned}
						\left(t-\tfrac{1}{2}-\varepsilon\right)\!\int_{\mathcal{M}} e^{-f}\,\mathtt{P}^{t-2}\,v^{2-m}\,|\nabla\mathtt{P}|^2\,\varphi^\theta&+\left(\tfrac{m-d}{m}-\varepsilon\right)\!\int_{\mathcal{M}} e^{-f}\,\mathtt{P}^{t+1}\,v^{1-m}\,\varphi^\theta\\&\leq \mathtt{K}\!\int_{\mathcal{M}} e^{-f}\,v^{2-m+t}\,|\nabla \varphi|^{2(t+1)}\,\varphi^{\theta-2(t+1)}\,.    
					\end{aligned}
				\end{equation} 
				Now we may conclude by taking $t=m-2$, and since $m>d\geq3$, such a choice is valid in the sense of \eqref{t-cond}.
				Now using \eqref{def-cutoff} and \eqref{bis-gro-inf}, and choosing $\theta$ sufficiently large in terms of $m$, the right-hand side of \eqref{quasi} can be estimated as follows:
				\begin{equation*}
					\int_{\mathcal{M}} e^{-f}\,v^{2-m+t}\,|\nabla \varphi|^{2(t+1)}\,\varphi^{\theta-2(t+1)}\leq \int_{\mathcal{M}} e^{-f}\,|\nabla \varphi|^{2m-2}\leq C\,R^{2-2m+d}\,,
				\end{equation*}
				for a constant $C>0$ depending only on $d,m,\varepsilon$. Since $m>\frac{d+2}{2}$, it follows that the right-hand side converges to $0$ as $R\to+\infty$. Combining this with \eqref{quasi} and choosing a suitable $\varepsilon>0$ in terms of $d$ and $m$ yields $\mathtt{P}\equiv0$, which is a contradiction with $v>0$. So $v=0$, which is a contradiction with $u>0$. 
			\end{proof}
			
			\subsection{Proof of Theorem \ref{thm-infty-crit}}
			\begin{proof}
				As before, from the maximum principle for $\mathcal{L}$, either $u=0$ everywhere or $u>0$ everywhere. Let us assume the latter case, so it follows that $v>0$ and $\mathtt{P}>0$. 	
				
				According to the hypotheses of the theorem, a volume comparison theorem always holds with dimensional constant $d$. Therefore, in this proof, we may freely apply Corollary \ref{int-est-cor} and estimates \eqref{int-est-1}--\eqref{int-est-2}, with $n$ replaced by $d$.
				
				\medskip
				\noindent{\textbf{Preliminary step:}}				
				As in the proof of Theorem \ref{thm-leq-crit}, under certain assumptions on $u$ or the dimension $d$, we will derive that {$\nabla \mathtt{P}\equiv0$}. 
				Indeed, if $\mathtt{P}\equiv c>0$, from \eqref{ana-geo-eq} it follows that $k[v]=0$, which from $\eqref{eq-k-inf}$ implies
				\begin{align}
					&\nabla^2v=\frac{\Delta v}{d}g\,,\label{imp-a-inf}\\
					&\nabla f\cdot\nabla v\equiv0\,,\label{imp-c-inf}\\
					&\mathrm{Ric}_{\infty,d}\left(\nabla v,\nabla v\right)=0\label{imp-d-inf}\,.
				\end{align}
				Once again, \eqref{imp-d-inf} is not used in the sequel.
				By \eqref{imp-c-inf}, it follows that 
				$\Delta v=\mathcal{L}v=\mathtt{P}\equiv c$. Plugging this into \eqref{imp-a-inf} and using \cite[Theorem 2]{Tashiro}, it follows that $(\mathcal{M},g,\mu)$ is isometric to $\R^d$ with the standard metric. It now follows that $v$ has the form
				\begin{equation}\label{prelim-form-v-inf}
					v(x)=\frac{c}{2d}|x-o|^2+b\,,
				\end{equation}
				for some non-negative constant $b$ and some $o\in\mathcal{M}$. By plugging \eqref{prelim-form-v-inf} back into \eqref{eq-for-v} (and recalling that $m=d$), we find that $b=\frac{2}{d-2}\frac{1}{c}$, and converting back to $u$, it turns out that $u$ is precisely \eqref{at-bubble}. 
				
				Finally, we show that the weight $f$ is trivial by plugging \eqref{prelim-form-v-inf} into \eqref{imp-c-inf}. We derive
				\begin{equation*}
					r\,f_r=0\quad\mathrm{in}\;\;\R^d\,,
				\end{equation*}
				where the radii are taken with respect to the point $o\in\mathcal{M}=\R^d$, which, together with the stated regularity of $f$ implies $\nabla f\equiv0$.
				
				From now on we proceed similarly to the proof of Theorem \ref{thm-leq-crit}, using Proposition \ref{int-key-pro} (recall that since $\nabla f\cdot\nabla v\geq0$, it holds $W_{\!f}[v]\geq0$) to prove that $\nabla \mathtt{P}\equiv0$ under suitable assumptions.
				
				\medskip
				\noindent{\textbf{Proof of i), case $\boldsymbol{d=3}$.}}
				We consider \eqref{int-key-pro-eq} with $t=\frac{1}{2}+\delta$, $\delta\in(0,\frac{1}{2})$, and $\varepsilon<\delta$. Since $d=3$ satisfies
				\begin{equation*}
					0\leq \frac{d-2-t}{1-t} \leq \frac{d}{2}+1\,,
				\end{equation*}
				\emph{i.e.} \eqref{n-cond-2} with $n$ replaced by $d$, we may proceed exactly as in the proof of Theorem \ref{thm-leq-crit}; in particular, we simply redo the argument from \eqref{n-cond-2} to \eqref{r1} with $n$ replaced by $d$. Along the way, we crucially use \eqref{int-est-1}--\eqref{int-est-2} with $n$ replaced by $d$. We obtain a modified version of \eqref{r1}, which, inserted back into \eqref{int-key-pro-eq} yields exactly \eqref{pre-r-inf} with $n=d$.
				Taking $R\to+\infty$ concludes the argument.
				
				\medskip
				\noindent{\textbf{Proof of i), case $\boldsymbol{d=4}$.}}
				Note that $d=4$ satisfies 
				\begin{equation*}
					2\leq d-2+t<\frac{d}{2}+1\,.
				\end{equation*}
				Therefore, we may rewrite all of the computations from \eqref{int-key-pro-eq-appl-bis} to \eqref{pre-r-inf} with $n$ replaced by $d$. Taking $R\to+\infty$ again finishes the argument. 
				
				\medskip
				\noindent{\textbf{Proof of i), case $\boldsymbol{d=5}$.}}
				The exact same critical argument used in the proof of Theorem \ref{thm-leq-crit} $i)$ with $n=5$ also works in our case, just by suitably replacing $n$ with $d$. Once again, we find that \eqref{r2} holds with $n=d=5$ in the constant. We conclude by taking $R\to+\infty$.  
				
				\medskip
				\noindent{\textbf{Proof of ii)}}
				We proceed as in the Proof of Theorem \ref{thm-leq-crit} $iii)$. The key ingredients are once again the Yau estimate Lemma \ref{yau-lem} - which is valid thanks to the assumptions \eqref{dot-cond} and \eqref{lapl-comp-inf-eq} (actually \eqref{rough-lapl-comp}) - and \eqref{pw-decay-cond-d}.
				
				\medskip
				\noindent{\textbf{Proof of iii)}}
				Now we adapt the Proof of Theorem \ref{thm-leq-crit} $iv)$.
				First, under the assumption \eqref{lapl-comp-inf-eq}, and in view of the discussion around \eqref{bis-gro-inf}, \eqref{global-energy-claim} holds with $n=d$. Next, since \eqref{lapl-comp-inf-eq} again holds, we may use Lemma \ref{super-harm-lem} with $n=d$. These elements allow one to reproduce the proof, 
				thus we find that $\nabla\mathtt{P}\equiv0$ and $\nabla f\cdot\nabla v\equiv0$.
			\end{proof}
			
			\subsection{Proof of Theorem \ref{thm:exp-class}}\label{subsec:liou-proof}
			Let us quickly explain why making the stronger assumption $\mathrm{Ric}_{n,2}\geq0$ for some $n>2$ does not in any case allow us to drop the additional condition \eqref{dot-cond}. Indeed, let us consider \eqref{eq-k-inf} with $m=d=2$ and use $\mathrm{Ric}_{\infty,2}\,(\nabla v,\nabla v)=\mathrm{Ric}_{n,2}\,(\nabla v,\nabla v)+\frac{|\nabla f\cdot\nabla v|^2}{n-2}$. In this case, \eqref{eq-k-inf} becomes
			\begin{equation*}
				\mathsf{k}[v]=\norm{\nabla^2v-\frac{\Delta v}{2}g\,}^2_{\mathrm{H.S.}}\!+\mathcal{L}v\,\nabla f\cdot\nabla v+\frac{n}{2(n-2)}\left|\nabla f\cdot\nabla v\right|^2+\mathrm{Ric}_{n,2}\left(\nabla v,\nabla v\right).
			\end{equation*}
			{Due to the presence of the second (possibly sign-changing) term on the right-hand side}, in order to deduce a series of consequences from $\mathsf{k}[v]=0$, it is still necessary to assume \eqref{dot-cond}, \emph{i.e.} $\nabla f\cdot\nabla v\geq0$.
			
			\begin{proof}
				Our goal will be to show that $\mathtt{P}\equiv c$ for some $c>0$. From there, under the crucial assumption \eqref{dot-cond}, we may reuse the Preliminary step of the Proof of Theorem \ref{thm-infty-crit} to conclude that $v=e^{-\frac{u}{2}}$ is precisely \eqref{prelim-form-v-inf}, $\nabla f\equiv0$ and that $(\mathcal{M},g,\mu)$ is isometric to $\R^2$ with the standard metric. Converting $v$ back into $u$ gives that $u$ is equal to \eqref{log-bubble}. Before we begin, let us also recall that we assume that the volume growth comparison result \eqref{bis-gro-inf} holds for some $o\in\mathcal{M}$.
				
				To this end, let us begin by taking Lemma \ref{ibp-lem} with $q=1$, $m=2$, and $\psi=\varphi_R^2$, where $\varphi_R$ is the standard cut-off function centered at $o\in\mathcal{M}$ defined in Lemma \ref{cutoff-lem}. Indeed, applying \eqref{ibp-formula-1} and subsequently applying Young's inequality yields
				\begin{equation*}
					\begin{aligned}
						\int_{\mathcal{M}}e^{-f}\,v^{-1}\,\left|\nabla v\right|^2\,\varphi_R^2+\frac{1}{2}\int_{\mathcal{M}} e^{-f}\,v^{-1}\,\varphi_R^2
						&=-2\int_{\mathcal{M}} e^{-f}\,\varphi_R\,(\nabla v\cdot\nabla \varphi_R)\\
						&\leq\frac{1}{2}\int_{\mathcal{M}}e^{-f}\,v^{-1}\,|\nabla v|^2\,\varphi_R^2+\int_{\mathcal{M}}e^{-f}\,v\,|\nabla\varphi_R|^2\,.
					\end{aligned}
				\end{equation*}
				Rearranging terms and using the definition \eqref{eq-for-v} of $\mathtt{P}$, and the definition of $\varphi_R$, we obtain
				\begin{equation}\label{wc}
					\int_{\mathcal{B}_R(o)}e^{-f}\,\mathtt{P}\leq C\,R^{-2}\!\int_{\mathcal{A}_R(o)}e^{-f}\,v\,,
				\end{equation}
				for $C>0$ depending only on $d$.
				Now, we rewrite \eqref{pw-liou-cond} in terms of $v$
				\begin{equation*}
					v(x)\leq r^2(x)\,\mathcal{G}(r(x))\quad\mathrm{for}\;\mathrm{all}\;\;x\in\mathcal{M}\setminus\mathcal{B}_{R^*}(o)\,,
				\end{equation*}
				and plug this estimate into \eqref{wc} to obtain (for all $R>\frac{R^*}{2}$)
				\begin{equation}\label{wc2}
					\int_{\mathcal{B}_R(o)}e^{-f}\,\mathtt{P}\leq C\,\mathcal{G}(2R)\,\mu(\mathcal{A}_R(o))\leq C\,\mathcal{G}(2R)\,R^2\,,
				\end{equation}
				for some $C>0$ depending only on $d$ and the geometry of $\mathcal{M}$ in $\mathcal{B}_{R^*}(o)$.
				
				Also, a direct application of \eqref{fund-pw-ineq-eq} with $t\in[\frac{1}{2},1)$ and $m=d=2$ implies that
				\begin{equation}\label{pos-lapl}
					\mathcal{L}\,\mathtt{P}^t\geq0\,,\quad\forall t\in\left[\tfrac{1}{2},1\right),
				\end{equation}
				where once again we have crucially used the assumption that $\nabla f\cdot\nabla v\geq0$. Using the comments in \cite[p.~30]{PRS} and \cite[Proposition 1.3]{RS}, \eqref{g-cond}, \eqref{wc2}, and \eqref{pos-lapl} are enough to apply \cite[Theorem 1.1]{PRS} to determine that $\mathtt{P}^t\equiv c>0$. From this point, we may conclude as explained at the beginning of the proof.
			\end{proof}
			
			\section{Positive solutions with non-negative infinity curvature}\label{sec:bubble}
			In this section we construct critical and supercritical radial solutions of \eqref{leq}
			for $d\geq3$ 
			in radial \emph{weighted model manifolds} that satisfy $\mathrm{Ric}_{\infty,d}\geq0$ and have Euclidean volume growth rate \eqref{bis-gro-inf}, but \emph{do not} satisfy \eqref{dot-cond}, thus showing the optimality of the latter condition in Theorem \ref{thm-infty-crit}.
			
			\subsection{The Lane-Emden equation on model manifolds}
			Let us take $(\mathcal{M},g,\mu)$ to be a model manifold. That is, $\mathcal{M}$ has a \emph{pole} $o\in\mathcal{M}$, and in polar coordinates $p\mapsto(r,\theta)$ for $r=\mathrm{dist}(p,o)$ and $\theta\in\mathbb{S}^{d-1}$, and the metric $g$ has the expression
			\begin{equation}\label{mod-metric}
				g=\mathrm{d}r^2+\psi^2(r)\,\mathrm{d}\theta^2\quad\mathrm{on}\;\;\mathcal{M}\setminus\left\{o\right\}\,,
			\end{equation}
			where $d\theta$ is the usual metric on $\mathbb{S}^{d-1}$.
			The function
			$$\psi\in C^\infty\!\left(\left[0,+\infty\right)\right)\,,$$
			satisfies {for any non-negative integer $k$
			\begin{equation}\label{psi-nec2}
			\psi^{(2k)}(0)=0\,,\quad\psi'(0)=1\,,\quad\psi'(r)>0\quad\mathrm{for}\;\mathrm{all}\;\;r>0\,.
			\end{equation}
			}
			These assumptions are sufficient to guarantee that $\mathcal{M}$ is globally diffeomorphic to $\R^d$, and in particular that $g$ in \eqref{mod-metric} can be extended to the pole $o\in\mathcal{M}$. For more information on such manifolds, see for example \cite[Section 3]{Grigoryan} or \cite[Section 1.8]{AMR-book}. In this context, the Laplacian operator is written as 
			\begin{equation*}
				\Delta = \frac{\partial^2}{\partial r^2}+(d-1)\frac{\psi'}{\psi}\frac{\partial}{\partial r}+\Delta_{\mathbb{S}^{d-1}}\,,
			\end{equation*}
			where $\Delta_{\mathbb{S}^{d-1}}$ is the Laplace-Beltrami operator on $\mathbb{S}^{d-1}$. We also assume that the weight $f$ only depends on $r$ and is smooth up to $r=0$.
			
			Now, since we are looking for \emph{radial} solutions of \eqref{leq-d-crit}, the problem reduces to finding positive solutions of 
			the semilinear ordinary differential equation
			\begin{equation}\label{ODE}
				\begin{cases}
					-u''-(d-1)\frac{\psi'}{\psi}u'=u^p-f'u'\qquad r>0\,,\\
					u(0)>0\,,\\
					u'(0)=0\,.
				\end{cases}
			\end{equation}
			
			We recall that we are concerned with solutions satisfying $\nabla f\cdot\nabla u>0$ (somewhere in $\mathcal{M})$, which, in view of the radial setting and the previous observation, reduces to 
			\begin{equation*}
				f'(r)<0\quad\mathrm{for}\;\mathrm{some}\;\;r>0\,. 
			\end{equation*}
			The $\infty$-Bakry-\'{E}mery Ricci tensor reduces to a function of $r$:
			\begin{equation}\label{ber}
				\mathrm{Ric}_{\infty,d}=\underbrace{\left[-(d-1)\frac{\psi''}{\psi}+f''\right]}_{\mathrm{Ric}_{\infty,d}^{\mathrm{r}}}dr^2+\underbrace{\left[-{\psi''\psi}+(d-2)\!\left({1-\left(\psi'\right)^2}\right)+{\psi\,\psi'f'}\right]}_{\mathrm{Ric}_{\infty,d}^{\theta}}d\theta^2\,.
			\end{equation}
			See \cite[eq. (2)]{Brendle} and \cite[p. 634]{PR} for a justification of this formula. We recall that we will be constructing a solution with $\mathrm{Ric}_{\infty,d}\geq0$, so it will sufficient for \emph{both} $\mathrm{Ric}_{\infty,d}^{\mathrm{r}}$ and $\mathrm{Ric}_{\infty,d}^{\theta}$ to be non-negative.

			\subsection{Constructing a global positive solution - Proof of Theorem \ref{constr-sol-theo}}
			Let us make a few preliminary observations on the solvability and properties of solutions of \eqref{ODE}.
			
			\medskip
			\noindent{\textbf{Claim 1: local existence and monotonicity.}}
			Since $u(0)>0$, it is possible using standard methods to construct a local positive solution 
			of \eqref{ODE} for $r\in(0,R)$, for some $R>0$. 
			
			Next, let us notice that the equation in \eqref{ODE} can be rewritten in divergence form as
			\begin{equation}\label{ODE-var}
				-\left(e^{-f}\,\psi^{d-1}\,u'\right)'=e^{-f}\,\psi^{d-1}u^p\,.
			\end{equation}
			By integrating, it follows that 
			\begin{equation*}
				u'(r)=-e^{f(r)}\,\psi^{1-d}(r)\int_0^r e^{-f}\,\psi^{d-1}\,u^p\,ds<0\,,
			\end{equation*}
			so $u$ is always decreasing in its interval of existence. 
			
			\medskip
			\noindent{\textbf{Claim 2: energy is decreasing.}}
			First, under the assumption 
			\begin{equation}\label{f-cond-1}
				f'<(d-1)\frac{\psi'}{\psi}\iff \left(e^{-f}\,\psi^{d-1}\right)'>0\iff\mathcal{L}r>0\,,
			\end{equation}
			it is easy to rule out energy singularities developing in finite time. Indeed, using \eqref{ODE}, it holds
			\begin{equation*}
				\mathcal{E}'_u(r)=u'(r)\,u''(r)+u(r)^p\,u'(r)=u'(r)^2\left(-(d-1)\frac{\psi'}{\psi}+f'\right)<0\,,
			\end{equation*}
			under \eqref{f-cond-1}. From now on, we assume \eqref{f-cond-1}.
			
			\medskip
			\noindent{\textbf{Claim 3: the solution does not vanish with its derivative.}}
			Next, let us confirm that $u'(R)=u(R)=0$ is not a possibility for any $R>0$. To see this, an integration of \eqref{ODE-var} yields
			\begin{equation*}
				0=\int_0^Re^{-f}\,\psi^{d-1}\,u^p\,ds\,,
			\end{equation*}
			which is incompatible with $u>0$ in $(0,R)$.
			
			\medskip
			\noindent{\textbf{Claim 4: the solution does not vanish with a negative derivative.}}
			Finally, to prove the existence of a positive solution of \eqref{ODE}, it is now enough to rule out the case that $u(R)=0$ and $u'(R)<0$ for some $R>0$. 
			
			In what follows, we adapt the strategy used in \cite{BFG,BGGV,MurSoa}. That is, consider the energy function
			\begin{equation*}
				\mathcal{E}_u(r)=\frac{u'(r)^2}{2}+\frac{u(r)^{p+1}}{p+1}\,,
			\end{equation*}
			and the Pohozaev/Lyapunov function
			\begin{equation}\label{def-pur}
				\mathcal{P}_u(r)=\left(\int_0^re^{-f(s)}\,\psi(s)^{d-1}\,ds\right)\,\mathcal{E}_u(r)+\frac{1}{p+1}\,e^{-f(r)}\,\psi(r)^{d-1}\,u(r)\,u'(r)\,.    
			\end{equation} 
			We use these two functions to extend the locally-defined positive solution of \eqref{ODE} for all $r>0$.
			Notice that $u(R)=0$ and $u'(R)<0$ implies that $\mathcal{P}_u(R)>0$, so if we can show that $\mathcal{P}_u(r)\leq0$ for all $r$, we have obtained a contradiction. That is, the local solution $u$ can be extended for all $r>0$. 
			
			To this end, we calculate the derivative\footnote{From now on, we drop references to the argument $r$ where possible in order to lighten the notation.} 
			\begin{equation}\label{deriv-p-1}
				\mathcal{P}'_u(r)=\underbrace{\left[\left(\frac{1}{2}+\frac{1}{p+1}\right)\,e^{-f}\,\psi^{d-1}-\left(e^{-f}\psi^{d-1}\right)'\left(\int_0^re^{-f}\,\psi^{d-1}\,ds\right)\right]}_{\mathcal{K}(r)}\,u'(r)^2\,.
			\end{equation}
			Therefore, to construct a global positive solution, it is enough to show that $\mathcal{K}(r)\leq0$ for all $r$.
			
			It is now necessary to calculate $\mathcal{K}(r)$, and in particular $\int_0^re^{-f}\,\psi^{d-1}\,ds$, which is (up to a constant) the weighted volume of $\mathcal{B}_r$. To this end, we begin by making the calculation
			\begin{equation*}
				\begin{aligned}
					\int_0^re^{-f}\,\psi^{d-1}&=\int_0^r\left(e^{-f}\,\psi^{d-1}\right)'\frac{e^{-f}\,\psi^{d-1}}{\left(e^{-f}\,\psi^{d-1}\right)'}\\
					&=\frac{\left(e^{-f}\,\psi^{d-1}\right)^2}{\left(e^{-f}\,\psi^{d-1}\right)'}-\int_0^re^{-f}\,\psi^{d-1}\left(\frac{e^{-f}\,\psi^{d-1}}{\left(e^{-f}\,\psi^{d-1}\right)'}\right)'\\
					&=\frac{\left(e^{-f}\,\psi^{d-1}\right)^2}{\left(e^{-f}\,\psi^{d-1}\right)'}-\int_0^re^{-f}\,\psi^{d-1}+\int_0^r\left(\frac{e^{-f}\,\psi^{d-1}}{\left(e^{-f}\,\psi^{d-1}\right)'}\right)^2\left(e^{-f}\,\psi^{d-1}\right)''\,,
				\end{aligned}
			\end{equation*}
			where we have crucially used the assumption \eqref{f-cond-1}. 
									
			From this point, using the definition \eqref{ber} of the radial part of the $\mathrm{Ric}_{\infty,d}$ curvature tensor, we may continue the calculation to eventually find that 
			\begin{equation}\label{final-k}
				\begin{aligned}
					\mathcal{K}(r)&=\left(\frac{1}{2}+\frac{1}{p+1}-\frac{d-1}{d}\right)\,e^{-f}\,\psi^{d-1}\\
					&\quad +\frac{d-1}{d}\frac{\left(e^{-f}\,\psi^{d-1}\right)'}{e^{-f}\,\psi^{d-1}}\int_0^r\left(\frac{e^{-f}\,\psi^{d-1}}{\left(e^{-f}\,\psi^{d-1}\right)'}\right)^2e^{-f}\,\psi^{d-1}\left(\mathrm{Ric}_{\infty,d}^{\mathrm{r}}+\frac{2}{d-1}\mathcal{L}r\,f'+\frac{(f')^2}{d-1}\right)ds\,.
				\end{aligned}
			\end{equation}
			Note that \eqref{final-k} is consistent with the formulas established in \cite[pp. 769--770]{MurSoa} for the unweighted case.
			Therefore, to find a global solution of \eqref{ODE} (by proving that $\mathcal{K}(r)\leq0$), it is enough to find a pair $(\psi,f)$ satisfying
			\begin{equation*}
				\mathrm{Ric}_{\infty,d}^{\mathrm{r}}+\frac{2}{d-1}\mathcal{L}r\,f'+\frac{(f')^2}{d-1}\leq0\,,
			\end{equation*}
			which, in view of \eqref{ber} and the definition of $\mathcal{L}$, can be rewritten as 
			\begin{equation}\label{infty-exist-cond}
				-(d-1)\frac{\psi''}{\psi}+f''+2\frac{\psi'\,f'}{\psi}-\frac{(f')^2}{d-1}\leq0\,.
			\end{equation}
			Also, note that the constant $\frac{1}{2}+\frac{1}{p+1}-\frac{d-1}{d}$ is negative for $p>p_S(d)$ and vanishes for $p=p_S(d)$, so, if we can find a pair $(\psi,f)$ satisfying \eqref{infty-exist-cond}, the manifold $(\mathcal{M},g,\mu)$ admits \emph{both critical and supercritical solutions} of \eqref{leq}.
			
			Let us now summarize the conditions that we request $\psi$ and $f$ to satisfy in order to have the global existence of a positive solution to \eqref{ODE}, and $\mathrm{Ric}_{\infty,d}>0$ in $\mathcal{M}\setminus\{o\}$:
			\begin{enumerate}[i)]
				\item $0< \mathrm{Ric}_{\infty,d}^{\mathrm{r}}:=-(d-1)\frac{\psi''}{\psi}+f''$, for $r>0$,
				\item $0< \mathrm{Ric}_{\infty,d}^\theta:=-{\psi''\,\psi}+(d-2)\!\left(1-\left(\psi'\right)^2\right)+{\psi\,\psi'f'}$, for $r>0$,
				\item $\mathrm{Ric}_{\infty,d}^{\mathrm{r}}\leq-2\frac{\psi'\,f'}{\psi}+\frac{(f')^2}{d-1}$, (so $u$ is globally positive in view of \eqref{deriv-p-1}--\eqref{infty-exist-cond})\,.
			\end{enumerate}
			
To construct an example of $(\psi,f)$ satisfying $\mathrm{i)}$--$\mathrm{iii)}$, let us consider any $\psi$ such that \eqref{psi-nec2} holds, and that
\begin{equation} \label{psi-nec3}
	\psi'''(0) < 0 \quad\mathrm{and}\;\;
\psi'' (r) < 0 \;\;\mathrm{for}\;\mathrm{all}\;\;r>0\,.
\end{equation}
Furthermore, we request that
\begin{equation} \label{psi-nec4}
 \psi(r) = \alpha\, r + o (r) \quad \mathrm{and}\;\; \int_0^r\psi''(s)\,\psi(s)\,ds = o(r^{1-\varepsilon})\;\;\mathrm{as}\;\;r\to+\infty\,,
\end{equation}
for some $\alpha\in(0,1)$ and $\varepsilon >0$. There is an explicit example that satisfies \eqref{psi-nec2} and \eqref{psi-nec3}--\eqref{psi-nec4}:
			\begin{equation*}%\label{def-psi}
				\psi(r)=\alpha\, r+(1-\alpha)\frac{r}{\sqrt{r^2+1}}\quad\mathrm{for}\;\mathrm{some}\,\alpha\in(0,1)\,.
			\end{equation*} 
			Next, we set $f$ to satisfy
			\begin{equation}\label{ode-rel}
				f''+2\frac{\psi'}{\psi}f'=(d-1)\frac{\psi''}{\psi}\,,
			\end{equation}
			so $\mathrm{iii})$ above is immediately satisfied. Furthermore, \eqref{ode-rel} is a first-order linear differential equation in the variable $f'$, so there is a whole family of solutions
			\begin{equation}\label{def-f-expl}
				f(r)=f(0)+(d-1)\int_0^r\frac{\int_0^s\psi''(\tau)\,\psi(\tau)\,d\tau}{\psi^2(s)}\,ds\qquad\mathrm{for}\;\mathrm{all}\;\;r>0\,.
			\end{equation}
			Since $\psi''<0$ for all $r>0$, from \eqref{def-f-expl} we have $f'<0$ for all $r>0$. 
			From $f'<0$ and \eqref{ode-rel}, $\mathrm{i)}$ holds. This also implies that $\nabla f\cdot \nabla u>0$ in $\mathcal{M}\setminus\{o\}$.

			Now it is only left to check $\mathrm{ii)}$. In fact, by the sign of $\psi''$, it is enough to check that
			\begin{equation*} 
				\Psi:= (d-2)(1-\left(\psi'\right)^2)+\psi \psi'f'>0\,.
			\end{equation*}
			Taking first the derivative and then applying \eqref{ode-rel}, we derive
			\begin{align*}
				\Psi' &= -2(d-2)\,\psi'\psi''+\psi''\psi f'+(\psi')^2f'+\psi'\psi f''\\
				&=-(d-3)\psi'\psi''-(\psi')^2f'+\psi''\psi f' > 0\,.
			\end{align*} 
			In view of $f'<0$ and $\psi''<0$ it follows that $\Psi$ is increasing and 
			\begin{equation*}
			\Psi (r) > \Psi(0) = 0
			\end{equation*}
			This implies that $\mathrm{ii)}$ is satisfied.

\medskip
\noindent{\textbf{Conclusion: Proofs of further properties.}}\\
It is still left to prove properties 2), 3), and 7) of Theorem \ref{constr-sol-theo}. First, we notice that 
\begin{equation*}
	f'(r) = \frac{d-1}{3}\, \psi'''(0)\,
	r + o(r) \quad \textmd{as } r \to 0 \,,
\end{equation*}
and, from \eqref{psi-nec4},
\begin{equation*}
	f'(r) =  o(r^{-1-\varepsilon}) \quad \textmd{as } r \to +\infty \,,
\end{equation*}
which implies that $f$ is bounded. Now 2) (\emph{i.e.} \eqref{bis-gro-inf}) easily follows. Indeed, let $C_1$, $C_2$ be such that 
\begin{equation}\label{simple-bound-f}
	0<C_1<e^{-f(r)}<C_2\quad\mathrm{for}\;\mathrm{all}\;\;r>0\,.
\end{equation}
Furthermore, by construction, it holds that 
\begin{equation}\label{simple-bound-psi}
	\alpha\, r< \psi(r)< r\quad\mathrm{for}\;\mathrm{all}\;\;r>0\,.
\end{equation}
By these estimates, 
\begin{equation*}
	\mathrm{vol}_{f}(\mathcal{B}_r(o))=\int_0^r\,e^{-f}\,\psi^{d-1}\leq \tfrac{C_2}{d}\,r^d\,.
\end{equation*}

Next, we show 3), \emph{i.e.} that \eqref{lapl-comp-inf-eq} does not hold, but \eqref{rough-lapl-comp} does. To this end, from \eqref{def-f-expl} and integration by parts, we have 
\begin{equation}\label{lapl-rad-eq}
\begin{aligned}
	\mathcal{L}r&=(d-1)\frac{\psi'(r)}{\psi(r)}-f'(r)\\
	&=(d-1)\frac{\psi'(r)}{\psi(r)}-(d-1)\frac{\int_0^r\psi''(s)\,\psi(s)\,ds}{\psi^2(r)}\\
	&=(d-1)\frac{\int_0^r(\psi'(s))^2\,ds}{\psi^2(r)}\,,
\end{aligned}
\end{equation}
so \eqref{lapl-comp-inf-eq} holds if and only if $\chi(r):=\int_0^r(\psi'(s))^2\,ds-\frac{\psi^2(r)}{r}\leq 0$. However, it can be easily checked that 
\begin{equation*}
	\chi'(r)=\left(\psi'(r)-\frac{\psi(r)}{r}\right)^2>\left(\psi'-1\right)^2>0\quad\mathrm{for}\;\mathrm{all}\;r>0\,,
\end{equation*}
and $\chi(0)=0$,
where we have used \eqref{simple-bound-psi} and $\psi''<0$. Therefore $\chi(r)>0$ for all $r>0$ and \eqref{lapl-comp-inf-eq} fails. 

Although \eqref{lapl-comp-inf-eq} does not hold, the weaker estimate \eqref{rough-lapl-comp} \emph{does holds} for some $C>0$. This is once again a simple consequence of \eqref{simple-bound-f}--\eqref{simple-bound-psi}. Indeed, from \eqref{lapl-rad-eq}, we have\footnote{In view of Lemma \ref{super-harm-lem}, the following estimate implies an asymptotic lower bound on $u$.} 
\begin{equation*}
	\mathcal{L}r\leq \frac{d-1}{\alpha^2 r}\,.
\end{equation*}

Finally, we prove 7) - the asymptotic upper bound \eqref{asy-upp-bd-u} - by adapting \cite[Proof of Theorem 1.4 i)]{MurSoa}. We exploit the fact that $\mathcal{P}_u(r)\leq0$ defined in \eqref{def-pur} (which follows from \eqref{deriv-p-1} and $\mathcal{K}(r)\leq0$). Our starting point is to discard the term $\frac{u'(r)^2}{2}$ in $\mathcal{E}_u(r)$ and rearrange some terms: we obtain
\begin{equation*}
	\frac{-u'(r)}{u^p(r)}\geq\frac{\int_0^re^{-f(s)}\psi(s)^{d-1}\,ds}{e^{-f(r)}\psi(r)^{d-1}}\,.
\end{equation*}
The left-hand side is equal to $\frac{1}{p-1}(u(r)^{1-p})'$, whereas the right-hand side is estimated using \eqref{simple-bound-f}--\eqref{simple-bound-psi}.
This yields the estimate
\begin{equation*}
	\tfrac{1}{p-1}(u(r)^{1-p})'\geq\tfrac{C_1}{C_2}\tfrac{\alpha^{d-1}}{d}\,r \quad\mathrm{for}\;\mathrm{all}\;\;r>0\,,
\end{equation*}
integrating this inequality from $0$ to $r$ and suitably rearranging yields \eqref{asy-upp-bd-u} with $C=\frac{p-1}{2d}\frac{C_1}{C_2}\alpha^{d-1}$.

			\appendix
			\section{Weighted Laplacian comparison and consequences}\label{appen:comp}
			In this section, we consider some consequences of the weighted Laplacian comparison inequality
			\begin{equation}\label{lapl-comp-n-eq}
				\mathcal{L}(\mathrm{dist}(x,o))\leq \frac{\kappa-1}{\mathrm{dist}(x,o)}\quad\mathrm{weakly}\;\mathrm{in}\;\;\mathcal{M}\,,
			\end{equation}
			for some $o\in\mathcal{M}$ and some $\kappa>1$. 
			
			\begin{remark}[On the distance function and the cut locus]\label{dist-rem}
				On Riemannian manifolds the distance function from a point $o\in\mathcal{M}$ is globally Lipschitz continuous (with constant $1$) and smooth away from the cut locus of $o$ in $\mathcal{M}$, denoted by $\mathrm{cut}(o)$. Since the cut locus is a zero-measure set, \eqref{lapl-comp-n-eq} thus holds a.e. $x\in\mathcal{M}$. 
			\end{remark} 
			
			\normalcolor
			The following statement combines well-known results from \cite{BQ, Qian,WeiWylie}.
			\begin{lemma}
				Let $(\mathcal{M},g,\mu)$ be a complete weighted
				Riemannian manifold of dimension $d\geq2$. If $\mathrm{Ric}_{n,d}\geq0$ for some $n\geq d$, then for all $o\in\mathcal{M}$, \eqref{lapl-comp-n-eq} holds with $\kappa=n$
				and pointwise in $\mathcal{M}\setminus\mathrm{cut}(o)$,
				which implies
				\begin{equation}\label{bis-gro}
					\int_{\mathcal{B}_R(o)}e^{-f}\,\leq C\,R^n\quad\mathrm{for}\;\mathrm{all}\;R\geq r\,,
				\end{equation}
				for all $r>0$, where $C>0$ depends only on $d$, $n$ and the geometry of $\mathcal{M}$ in $\mathcal{B}_r(o)$. 
				
				If $\mathrm{Ric}_{\infty,d}\geq0$ and there exists a point $o\in\mathcal{M}$ such that $f'\geq0$ along all geodesics starting at $o$, then \eqref{lapl-comp-n-eq} holds with respect to $o$ for all $r>0$ and with $\kappa=d$. As a consequence, \eqref{bis-gro} holds in this case for $n=d$ and for $C>0$ depending only on $d$ and the geometry of $\mathcal{M}$ in $\mathcal{B}_r(o)$.
			\end{lemma}
			
			We state and prove a useful lemma regarding superharmonic functions in our generalized framework, the proof of which follows \cite[Lemma 2.8]{CM}.
			\begin{lemma}\label{super-harm-lem}
				Let $(\mathcal{M},g,\mu)$ be a complete, non-compact weighted Riemannian manifold of dimension $d\geq2$. Let $u\in C^2(\mathcal{M})$ be positive and $\mathcal{L}$-superharmonic, \emph{i.e.} $-\mathcal{L}u\geq0$, and assume that $(\mathcal{M},g,\mu)$ satisfies a Laplacian-type comparison inequality
				\eqref{lapl-comp-n-eq}
				for some $\kappa>2$ and $o\in\mathcal{M}$.
				Then for all $R>0$ there is a constant $A>0$ depending only on $R$, $\kappa$, and the behavior of $u$ on $\mathcal{B}_R(o)$ such that
				\begin{equation}\label{super-harm-ineq}
					u(x)\geq\frac{{A}}{\mathrm{dist}(x,o)^{\kappa-2}}\,,\quad\mathrm{for}\;\mathrm{all}\;\;x\in\mathcal{M}\setminus\mathcal{B}_R(o)\,.
				\end{equation}
			\end{lemma}
			\begin{proof}
				Let us fix $o\in\mathcal{M}$, and for a fixed $R>0$, we define
				\begin{equation*}
					\overline{u}(x)=u(x)-\frac{A}{\mathrm{dist}(x,o)^{\kappa-2}}\quad\mathrm{in}\;\;\mathcal{M}\setminus \mathcal{B}_R(o)\,,
				\end{equation*}
				where $A=R^{\kappa-2}\min_{\partial \mathcal{B}_R(o)}u$. If one applies \eqref{lapl-comp-n-eq},
				it requires just a simple calculation to determine that 
				\begin{equation*}
					-\mathcal{L}\overline{u}\geq0\quad\mathrm{weakly}\;\;\mathrm{in}\;\;\mathcal{M}\setminus \mathcal{B}_R(o)\,.
				\end{equation*}
				Also, by construction we have
				\begin{equation}\label{bound-pos}
					\overline{u}\geq0\quad\mathrm{on}\;\;\partial \mathcal{B}_R(o)\,.
				\end{equation}
				Furthermore, by the positivity of $u$, it holds that
				\begin{equation}\label{harm-asym}
					\liminf_{|x|\to+\infty}\overline{u}(x)\geq0\,.
				\end{equation} 
				
				Now, by the maximum principle applied to $\overline{u}$ on open annuli $\mathcal{B}_{\lambda R}(o)\setminus \overline{\mathcal{B}_R(o)}$ for $\lambda>1$, we have
				\begin{equation*}
					\min_{\overline{\mathcal{B}_{\lambda R}(o)}\setminus \mathcal{B}_R(o)}\overline{u}=\min\left(\overline{u}\big|_{\partial\mathcal{B}_R(o)},\;\overline{u}\big|_{\partial \mathcal{B}_{\lambda R}(o)}\right)\,.
				\end{equation*}
				By taking $\lambda\to+\infty$, and applying \eqref{bound-pos}--\eqref{harm-asym}, we find that 
				$\overline{u}\geq0$ in $\mathcal{M}\setminus\mathcal{B}_R(o)$, and the thesis follows. 
			\end{proof}
			
			\section{Cheng-Yau--type estimates}\label{app:yau}
			Before we begin with the Yau--type estimates, let us briefly discuss cutoff functions on $\mathcal{M}$. Since our discussion involves the cut locus of $o$, let us recall from Remark \ref{dist-rem} that this set always has measure zero in Riemannian manifolds.  
			\begin{lemma}[Cutoff functions]\label{cutoff-lem}
				Let $(\mathcal{M},g,\mu)$ be a complete, connected, non-compact, weighted Riemannian manifold of dimension $d\geq2$. Then, for all $o\in\mathcal{M}$ and all $R>0$, there is a non-negative function $\varphi_{R}\in C^{0,1}_c(\mathcal{M})$ satisfying
				\begin{equation}\label{def-cutoff}
					\begin{gathered}
						0\leq\varphi_R\leq 1\quad\mathrm{in}\;\;\mathcal{M}\,,\quad\varphi_{R}=1\quad\mathrm{in}\;\;\mathcal{B}_R(o)\,,\quad\varphi_{R}=0\quad\mathrm{in}\;\; \mathcal{M}\setminus \mathcal{B}_{2R}(o)\,,\\
						\left|\nabla\varphi_{R}\right|\leq \frac{C}{R}\,,\quad \left|\nabla \varphi_R\right|^2\leq C\,\frac{\varphi_R}{R^2}\quad\mathrm{in}\;\;\mathcal{M}\setminus \mathrm{cut}(o)\,,
					\end{gathered}
				\end{equation}
				for $C>0$ depending only on $d$. Furthermore, if \eqref{lapl-comp-n-eq} holds for some $o\in\mathcal{M}$ and {$\kappa>1$}, then for all $R>0$,
				\begin{equation}\label{spec-cutoff}
					\quad -\mathcal{L}\,\varphi_R\leq\frac{C}{R^2}\quad\mathrm{in}\;\;\mathcal{M}\setminus\mathrm{cut}(o)\,,
				\end{equation}
				for $C>0$ depending only on $d$ and {$\kappa$}.
			\end{lemma}
			\begin{remark}
				The existence of cutoff functions satisfying \eqref{def-cutoff} is classical - see \emph{e.g.} \cite[Lemma 1]{Karp} or \cite[Proof of Theorem 5.1]{Li}. 
				Notice that, in view of Remark \ref{dist-rem}, to state estimates involving derivatives of $\varphi_R$ we must avoid the cut locus of $o$ in $\mathcal{M}$, where the distance function fails to be smooth. 
			\end{remark}
			\begin{proof}[Proof of Lemma \ref{cutoff-lem}]
				As remarked above, it is only necessary to provide a proof of \eqref{spec-cutoff}. The cutoff function $\varphi_R$ may be constructed by means of the composition $\varphi_R(x)=\phi_R(\mathrm{dist}(x,o))$, where $\phi_R\in C^2\!\left([0,+\infty)\right)$ satisfies 
				\begin{gather*}
					\phi_R\equiv1\quad\mathrm{on}\;\;[0,R]\,,\quad\phi_R\equiv0\quad\mathrm{on}\;\;[2R,4R)\,,\\
					-\phi_R'\leq \frac{C}{R}\,,\quad -\phi_R''\leq \frac{C}{R^2}\,,\quad\mathrm{on}\;\;[0,4R)\,,
				\end{gather*}
				for some $C>0$ depending only on $d$. Let us set $r=\mathrm{dist}(x,o)$. Now, using \eqref{lapl-comp-n-eq}, it holds 
				\begin{equation*}
					\begin{aligned}
						-\mathcal{L}\,\varphi_R&=-\phi_R'\,\mathcal{L}r-\phi_R''\leq -\frac{(\kappa-1)\,\phi_R'}{R}-\phi_R''\leq\frac{C}{R^2}\,,
					\end{aligned}
				\end{equation*}
				as desired. This computation is valid only outside the cut locus.
			\end{proof}
			
			Note that when $n>d$, similar estimates to the following ones were proven in \cite{Lu}.
			\begin{lemma}\label{yau-lem}
				Let $(\mathcal{M},g,\mu)$ be a complete, non-compact, boundaryless weighted Riemannian manifold satisfying $\mathrm{Ric}_{n,d}\geq0$ for some $n\geq d$. If $u\in C^3(\mathcal{M})$ is positive solution of \eqref{leq-crit}, then for all $o\in\mathcal{M}$,
				\begin{equation}\label{yau-ineq}
					\sup_{\mathcal{B}_R(o)}\frac{|\nabla u|^2}{u^2}\leq C\left(\frac{1}{R^2}+\sup_{\mathcal{B}_{2R}(o)}u^{\frac{4}{n-2}}\right),
				\end{equation}
				where $C>0$ depends only on $d$ and $n$. 
				
				If instead we have $\mathrm{Ric}_{\infty,d}\geq0$, and \eqref{lapl-comp-n-eq} holds for some $\kappa>1$, and \eqref{dot-cond} holds, 
				then \eqref{yau-ineq} still holds, with $n$ replaced by $d$, and with $C>0$ now depending only on $d$ and $\kappa$.
			\end{lemma}
			\begin{proof}
				First of all, if $n=d$, the statement reduces trivially to \cite[Lemma 2.2]{FMM}, \cite[Theorem A.3]{CFP}. Therefore, we assume from now on that $n>d$ or $n=\infty$.
				
				\medskip
				\noindent{\textbf{Case $\boldsymbol{n>d}$}:}
				The strategy, which is an adaptation of the classical Bernstein technique, essentially consists in evaluating the function 
				\begin{equation*}
					\mathcal{Q}=\mathcal{G}\times\varphi_R\,,
				\end{equation*}
				at a point of global maximum $x_0$, where $\varphi_R$, defined in Lemma \ref{cutoff-lem}, is a standard cut-off function centered at $o\in\mathcal{M}$, and $\mathcal{G}$ is a positive function to be defined below\footnote{Note that from now on we are assuming that $x_0\notin\mathrm{cut}(o)$. If this is not the case, one may follow the classical approximation technique in \cite[Proof of Theorem 3]{Calabi} or \cite[Proof of Theorem 3]{CY}. See also \cite[Proof of Theorem 6.1]{PLi}.}. By such a construction, it follows that $x_0\in\mathcal{B}_{2R}(o)$, and
				\begin{equation}\label{Q-grad}
					0=\nabla\mathcal{Q}(x_0)=\nabla\mathcal{G}(x_0)\,\varphi_R(x_0)+\mathcal{G}(x_0)\,\nabla\varphi_R(x_0)\,,
				\end{equation}
				\begin{equation}\label{Q-lapl}
					\Delta\mathcal{Q}(x_0)\leq0\,.
				\end{equation}
				Now using \eqref{Q-grad}--\eqref{Q-lapl} and the definition of $\mathcal{L}$, it holds
				\begin{equation}\label{lq-est}
					\begin{aligned}
						0\geq\Delta\mathcal{Q}(x_0)&=\mathcal{LQ}(x_0)+\nabla f(x_0)\cdot\nabla\mathcal{Q}(x_0)\\
						&=\mathcal{LQ}(x_0)\\
						&=\mathcal{LG}(x_0)\,\varphi_R(x_0)+2\,\nabla\mathcal{G}(x_0)\cdot\nabla\varphi_R(x_0)+\mathcal{G}(x_0)\,\mathcal{L}\varphi_R(x_0)\,.
					\end{aligned}
				\end{equation}
				
				Now, let us define $\mathcal{G}$. To this end, we set $w=-\log u$ and $\beta=\frac{4}{n-2}$, and it is straightforward to verify from \eqref{leq} that
				\begin{equation}\label{eq-for-w}
					\mathcal{L}w=|\nabla w|^2+e^{-\beta w}=:\mathcal{G}\,,
				\end{equation}
				where, as indicated, $\mathcal{G}$ will denote the right-hand side of \eqref{eq-for-w}. Let us now estimate $\mathcal{L}\mathcal{G}$ from below, directly employing the generalized Bochner formula \eqref{BE-Bochner}. That is, (from now on taking for granted that we are evaluating all functions at $x_0$)
				\begin{equation}\label{big-ineq}
					\begin{aligned}
						\mathcal{LG}&=\mathcal{L}|\nabla w|^2+\mathcal{L}e^{-\beta w}\\
						&=2\norm{\nabla^2 w}^2_{\mathrm{H.S.}}+2\,\nabla w\cdot\nabla\mathcal{L} w +2\,\mathrm{Ric}_{\infty,d}(\nabla w,\nabla w)+\beta^2\,e^{-\beta w}\,|\nabla w|^2-\beta\,e^{-\beta w}\mathcal{L}w\\
						&\geq \tfrac{2}{d}(\Delta w)^2+2\,\nabla w\cdot\nabla\mathcal{G} +2\,\mathrm{Ric}_{n,d}(\nabla w,\nabla w)+\tfrac{2}{n-d}|\nabla f\cdot\nabla w|^2-\beta\,e^{-\beta w}\mathcal{G}\\
						&\geq\tfrac{2}{d}(\Delta w)^2+2\,\nabla w\cdot\nabla\mathcal{G} +\tfrac{2}{n-d}|\nabla f\cdot\nabla w|^2-\beta\,e^{-\beta w}\mathcal{G}\\
						&=\tfrac{2}{n}\mathcal{G}^2+2\,\nabla w\cdot\nabla\mathcal{G} +\tfrac{2(n-d)}{nd}\!\left(\mathcal{G}+\tfrac{n}{n-d}\nabla f\cdot\nabla w\right)^2-\beta\,e^{-\beta w}\mathcal{G}\\
						&\geq \tfrac{2}{n}\mathcal{G}^2+2\,\nabla w\cdot\nabla\mathcal{G} -\beta\,e^{-\beta w}\mathcal{G}\,,
					\end{aligned}
				\end{equation}
				where in the second-to-last line we have simply completed the square using the definition of $\mathcal{L}$, and in the last line we have discarded a positive term which will not be used below. Throughout, we have used $\mathcal{L}w=\mathcal{G}$ to simplify.
				
				Next, let us plug the above estimate into \eqref{lq-est}, multiply by $\varphi_R$, and also use \eqref{Q-grad} to rewrite $\nabla\mathcal{G}(x_0)$ (and divide by $\mathcal{G}=\frac{\mathcal{Q}}{\varphi_R}$). Using these ingredients yields
				\begin{equation*}
					\begin{aligned}
						\frac{2}{n}\mathcal{Q}&\leq 2\,\nabla w\cdot\nabla\varphi_R+2\frac{|\nabla\varphi_R|^2}{\varphi_R}-\mathcal{L}\varphi_R+\beta\,e^{-\beta w}\,\varphi_R\\
						&\leq \varepsilon\,|\nabla w|^2\,\varphi_R+\left(2+\frac{1}{\varepsilon}\right)\frac{|\nabla\varphi_R|^2}{\varphi_R}-\mathcal{L}\varphi_R+\beta\,e^{-\beta w}\,\varphi_R\,.
					\end{aligned}
				\end{equation*}
				We notice that by definition, $|\nabla w|^2\,\varphi_R\leq \mathcal{G}\,\varphi_R=\mathcal{Q}$. This gives
				\begin{equation*}
					\left(\frac{2}{n}-\varepsilon\right)\mathcal{Q}\leq \left(2+\frac{1}{\varepsilon}\right)\frac{|\nabla\varphi_R|^2}{\varphi_R}-\mathcal{L}\varphi_R+\beta\,e^{-\beta w}\,\varphi_R\,.
				\end{equation*}
				Now using the properties \eqref{def-cutoff}--\eqref{spec-cutoff} of $\varphi_R$, the definition $\mathcal{Q}=\mathcal{G}\,\varphi_R$ and recalling that we are evaluating at $x_0$, we have 
				\begin{equation}\label{bern-1}
					\mathcal{Q}(x_0)\leq C\left(\frac{1}{R^2}+e^{-\beta w(x_0)}\right),
				\end{equation}
				where $C>0$ depends only on $d$ and $n$. Now, since by construction $\mathcal{Q}=\mathcal{G}\,\varphi_R$ achieves its maximum at $x_0\in \mathcal{B}_{2R}(o)$, we have from \eqref{bern-1},
				\begin{equation*}
					\sup_{\mathcal{B}_R(o)}|\nabla w|^2\leq \sup_{\mathcal{B}_R(o)}\mathcal{G}\leq\mathcal{Q}(x_0)\leq C\left(\frac{1}{R^2}+\sup_{\mathcal{B}_{2R}(o)}e^{-\beta w}\right),
				\end{equation*}
				where we have used the definition of $\mathcal{G}$ on the left-hand side. The thesis follows simply by converting $w$ back into $u$.
				
				\medskip
				\noindent{\textbf{Case $\boldsymbol{n=\infty}$}:}
				We use the same strategy, only making a few modifications. Let us begin by studying the second line of \eqref{big-ineq}, and continue the chain of inequalities, keeping in mind that $\mathrm{Ric}_{\infty,d}\geq0$ holds, rather than the stronger condition $\mathrm{Ric}_{n,d}\geq0$ for some $n>d$. We have
				\begin{equation*}
					\begin{aligned}
						\mathcal{LG}&=2\norm{\nabla^2 w}^2_{\mathrm{H.S.}}+2\,\nabla w\cdot\nabla\mathcal{L} w +2\,\mathrm{Ric}_{\infty,d}(\nabla w,\nabla w)+\beta^2\,e^{-\beta w}\,|\nabla w|^2-\beta\,e^{-\beta w}\mathcal{L}w\\
						&\geq\tfrac{2}{d}(\Delta w)^2+2\,\nabla w\cdot\nabla\mathcal{G} -\beta\,e^{-\beta w}\mathcal{G}\\
						&\geq\tfrac{2}{d}\mathcal{G}^2+\tfrac{4}{d}\mathcal{G}\,\nabla f\cdot\nabla w+2\,\nabla w\cdot\nabla\mathcal{G} -\beta\,e^{-\beta w}\mathcal{G}\,.
					\end{aligned}
				\end{equation*}
				Under the assumption \eqref{dot-cond} and \eqref{eq-for-w}, it holds that $\mathcal{G}\,\nabla f\cdot\nabla w\geq0$, so that term may be discarded.
				We may now proceed exactly as in the case $n>d$ to conclude.
			\end{proof} 
			
			\Acknowledgments
The authors express their gratitude to S. Pigola, who kindly discussed with them a preliminary version of this paper.			
			
			A.F. has been partially supported by ANR EINSTEIN CONSTRAINTS: past, present, and future. Research project ANR-23-CE40-0010-02.
			G.C. and T.P. have been supported by the Research Project of the Italian Ministry of University and Research (MUR) PRIN 2022 “Partial differential equations and related geometric-functional inequalities”, grant number 20229M52AS\_004. G.C. and T.P. are members of the “Gruppo Nazionale per l'Analisi Matematica, la Probabilità e le loro Applicazioni” (GNAMPA) of the “Istituto Nazionale di Alta Matematica” (INdAM, Italy) and have been partially supported by the “INdAM - GNAMPA Project”, CUP \#E5324001950001\#.

		\end{document}